\documentclass{amsart}
\usepackage[colorlinks,citecolor=blue,urlcolor=black,linkcolor=black]{hyperref}
\usepackage{amssymb,xspace,cmap}

\newtheorem*{thma}{Theorem~A}
\newtheorem*{thmb}{Theorem~B}
\newtheorem*{thmc}{Theorem~C}
\newtheorem*{thmd}{Theorem~D}
\newtheorem{thm}{Theorem}[section]
\newtheorem{cor}[thm]{Corollary}
\newtheorem{prop}[thm]{Proposition}
\newtheorem{fact}[thm]{Fact}
\newtheorem{lemma}[thm]{Lemma}
\newtheorem{claim}{Claim}[thm]
\newtheorem{subclaim}{Subclaim}[claim]

\theoremstyle{definition}
\newtheorem{defn}[thm]{Definition}

\theoremstyle{remark}
\newtheorem{remark}{Remark}

\newcommand*\axiomfont[1]{\textsf{\textup{#1}}\xspace}
\newcommand\ch{\axiomfont{CH}}
\newcommand\FA{\axiomfont{FA}}
\newcommand\BA{\axiomfont{BA}}
\newcommand\zfc{\axiomfont{ZFC}}
\newcommand\zf{\axiomfont{ZF}}
\newcommand\gch{\axiomfont{GCH}}
\newcommand\sch{\axiomfont{SCH}}
\newcommand\sq{\sqsubseteq}
\newcommand\s{\subseteq}
\newcommand\br{\blacktriangleright}
\renewcommand{\restriction}{\mathbin\upharpoonright}
\renewcommand\mid{\mathrel{|}\allowbreak}
\newcommand\Mid{\mathrel{}\middle|\mathrel{}}
\newcommand\sd{\framebox[2.7mm][l]{$\diamondsuit$}\hspace{0.4mm}{}}
\newcommand\one{\mathbf{1}}
\DeclareMathOperator{\U}{U}
\DeclareMathOperator{\pr}{Pr}
\DeclareMathOperator{\id}{id}
\DeclareMathOperator{\cf}{cf}
\DeclareMathOperator{\cl}{cl}
\DeclareMathOperator{\Tr}{Tr}
\DeclareMathOperator{\tr}{tr}
\DeclareMathOperator{\im}{Im}
\DeclareMathOperator{\ns}{NS}
\DeclareMathOperator{\ind}{ind}
\DeclareMathOperator{\reg}{Reg}
\DeclareMathOperator{\card}{Card}
\DeclareMathOperator{\otp}{otp}
\DeclareMathOperator{\dom}{dom}
\DeclareMathOperator{\acc}{acc}
\DeclareMathOperator{\nacc}{nacc}
\DeclareMathOperator{\refl}{Refl}
\DeclareMathOperator{\ssup}{ssup}
\DeclareMathOperator{\supp}{supp}

\author{Chris Lambie-Hanson}
\address{Department of Mathematics, Bar-Ilan University, Ramat-Gan 5290002, Israel.}
\curraddr{Department of Mathematics and Applied Mathematics, Virginia Commonwealth University,
Richmond, VA 23284, USA}
\urladdr{http://people.vcu.edu/~cblambiehanso}

\author{Assaf Rinot}
\address{Department of Mathematics, Bar-Ilan University, Ramat-Gan 5290002, Israel.}
\urladdr{http://www.assafrinot.com}

\thanks{This research was partially supported by the Israel Science Foundation (grant \#1630/14).}

\begin{document}
\title[Knaster and friends I]{Knaster and friends I: \\ Closed colorings and precalibers}
\begin{abstract}
  The productivity of the $\kappa$-chain condition, where $\kappa$ is a regular,
  uncountable cardinal, has been the focus of a great deal of set-theoretic research.
  In the 1970s, consistent examples of $\kappa$-cc
  posets whose squares are not $\kappa$-cc were constructed by Laver, Galvin, Roitman
  and Fleissner. Later, $\zfc$ examples were constructed by Todorcevic, Shelah, and
  others. The most difficult case, that in which $\kappa = \aleph_2$, was
  resolved by Shelah in 1997.

  In this work, we obtain analogous results regarding the infinite productivity
  of strong chain conditions, such as the Knaster property. Among other results, for
  any successor cardinal $\kappa$, we produce a $\zfc$ example of a poset with
  precaliber $\kappa$ whose $\omega^{\mathrm{th}}$ power is not $\kappa$-cc.
  To do so, we carry out a systematic study of colorings satisfying a strong
  unboundedness condition. We prove a number of results indicating circumstances
  under which such colorings exist, in particular focusing on cases in which
  these colorings are moreover closed.
\end{abstract}
\date{\today}

\maketitle

\section{Introduction}

Questions about the productivity of the $\kappa$-chain condition for
regular, uncountable cardinals $\kappa$ have led to a great deal of
set-theoretic research. (For an overview, see \cite{paper18}.) A central
tool that arose in these investigations, implicit in work of Galvin
\cite{galvin_chain_conditions} and isolated by Shelah \cite{shelah_productivity},
is the following principle asserting the existence of rather complicated colorings.
(For unfamiliar notation, in particular our conventions regarding the expression
$[\mathcal{A}]^2$, see the Notation subsection at the end of the Introduction.)

\begin{defn}[Shelah, \cite{shelah_productivity}]\label{def_pr1}
  $\pr_1(\kappa, \kappa, \theta, \chi)$ asserts the existence of a coloring
  $c:[\kappa]^2 \rightarrow \theta$ such that
  for every  $\chi'<\chi$, every family $\mathcal{A}
  \subseteq [\kappa]^{\chi'}$ consisting of $\kappa$-many pairwise disjoint
  sets, and every $i < \theta$, there is $(a,b) \in [\mathcal{A}]^2$ such
  that $c[a \times b] = \{i\}$.
\end{defn}
The primary connection between this principle and the productivity of the
$\kappa$-chain condition stems from the fact that, if $\kappa$ is a regular cardinal
and $\pr_1(\kappa, \kappa, 2, \omega)$ holds, then the $\kappa$-chain condition
fails to be productive.

The work in this paper is motivated in large part by questions concerning
the \emph{infinite} productivity of the $\kappa$-chain condition and its strengthenings,
in particular the $\kappa$-Knaster condition. We introduce and study the following
principle, which plays a role in questions about the infinite productivity of
the $\kappa$-Knaster condition that is analogous to the role played by
$\pr_1(\kappa, \kappa, \theta, \chi)$ in questions about the productivity
of the $\kappa$-chain condition.

\begin{defn}\label{def_U}
  $\U(\kappa,\mu,\theta,\chi)$ asserts the existence of a coloring
  $c:[\kappa]^2 \rightarrow\theta$ such that for every  $\chi'<\chi$, every family
  $\mathcal A\s[\kappa]^{\chi'}$ consisting of $\kappa$-many pairwise disjoint sets,
  and every $i<\theta$, there exists $\mathcal B\in[\mathcal A]^\mu$ such that
  $\min(c[a\times b])>i$ for all $(a,b)\in[\mathcal{B}]^2$.
\end{defn}

\begin{remark} Note two conceptual differences between Definitions \ref{def_pr1} and \ref{def_U}:
  \begin{enumerate}
    \item  The second coordinate in the principle $\U(\kappa, \mu, \theta, \chi)$
      plays a different role from the second coordinate in the principle
      $\pr_1(\kappa, \lambda, \theta, \chi)$. This is the reason we choose to only
      define the case $\lambda=\kappa$.
    \item While $\Pr_1(\kappa,\kappa,\theta,\chi)$ implies $\Pr_1(\kappa,\kappa,\theta',\chi)$ for $\theta'<\theta$,
      the principle $\U(\ldots)$ offers no monotonicity in the third coordinate.
      Indeed, the instance $\U(\kappa,\kappa,\kappa,\kappa)$ is a trivial consequence of $\zf$.
  \end{enumerate}
\end{remark}

It is worth pointing out that certain instances of the above principle are implicit in previous works.
We mention a few examples here.
\begin{itemize}
  \item Implicit in the conclusion of \cite[Claim~4.9]{Sh:365} is the instance $\U(\kappa,2,\omega,\chi)$.
  \item Implicit in the proof of \cite[Claim~4.1]{Sh:572} is a proof of the fact
    that for every infinite regular cardinal $\lambda$, $\U(\lambda^+,2, \lambda, \lambda)$ holds.
  \item Implicit in \cite[Theorem~6.3.6]{todorcevic_book} is the statement that,
    for every infinite cardinal $\lambda$, $\U(\lambda^+,\lambda^+, \omega, \cf(\lambda))$ holds.
  \item Implicit in the proof of \cite[Lemma~3.4]{paper27} is the fact that any
    witness to $\pr_1(\kappa,\kappa,\theta,\chi)$ is also a witness to $\U(\kappa,2,\theta,\chi)$.
\end{itemize}

\subsection{Summary of results}

The results in this paper primarily fall into two classes. The first consists of
results asserting that, under appropriate circumstances, certain instances of $\U(\ldots)$
provably hold. The second consists of applications of $\U(\ldots)$ to questions
regarding the infinite productivity of strengthenings of the $\kappa$-chain condition
and generalizations of Martin's Axiom to higher cardinals. We preview some of the
prominent results here, beginning with those from the first class.

\begin{thma}
  Suppose that $\chi, \theta < \kappa$ are infinite cardinals. If either
  of the two following hypotheses holds, then $\U(\kappa, \kappa, \theta, \chi)$
  holds:
  \begin{enumerate}
    \item $\square(\kappa)$ holds; or
    \item there exists a non-reflecting stationary subset of $E^\kappa_{\geq \chi}$.
  \end{enumerate}
  In particular, if $\theta \leq \lambda$ are infinite, regular cardinals, then
  $\U(\lambda^+, \lambda^+, \theta, \lambda)$ holds.
\end{thma}

\begin{proof}
  This follows from Corollaries \ref{l23} and \ref{l24} and the fact that,
  if $\lambda$ is a regular cardinal, then $E^{\lambda^+}_\lambda$ is a
  non-reflecting stationary set.
\end{proof}

The previous result indicates that all possible instances of $\U(\ldots)$ hold at
successors of regular cardinals and also provides cases in which nontrivial instances
of $\U(\ldots)$ hold at successors of singular cardinals and inaccessible cardinals.
The next results provide further information in this direction.

\begin{thmb}
  If $\lambda$ is a singular cardinal and $\theta \le \lambda$ is an infinite
  cardinal, then any one of the following hypotheses implies that
  $\U(\lambda^+, \lambda^+, \theta, \cf(\lambda))$ holds:
  \begin{enumerate}
    \item $2^\lambda = \lambda^+$;
    \item $\refl({<}\cf(\lambda),  \lambda^+)$ fails;
    \item there is a closed witness to $\U(\lambda^+,2,\theta,2)$;
    \item $\cf(\theta) < \cf(\lambda)$ and $2^{\cf(\lambda)}< \lambda$;
    \item $\cf(\theta) = \cf(\lambda)$;
    \item $\cf(\theta) = \omega$.
  \end{enumerate}
\end{thmb}

\begin{proof}
  (1) follows from Theorem~\ref{ch},
  (2)--(4) follow from Theorem ~\ref{increasechi},
  (5) follows from Corollary~\ref{predecessor},
  and
  (6) follows from Corollary~\ref{omegacolors}.
\end{proof}

\begin{thmc}
  If $\kappa$ is an inaccessible cardinal, $\chi<\kappa$ is an infinite cardinal,
  and there is a stationary subset of $E^\kappa_{\geq \chi}$
  that does not reflect at any inaccessible cardinal, then $\U(\kappa, \kappa,
  \theta, \chi)$ holds for every infinite cardinal $\theta<\kappa$.
\end{thmc}

\begin{proof}
  This follows from Theorem~\ref{inaccessible_reflection}.
\end{proof}

Observe that the presence of large cardinals places limits on the extent to
which $\U(\ldots)$ holds at inaccessible cardinals or successors of singular cardinals.
In particular, it is immediate that, if $\kappa$ is weakly compact, then
$\U(\kappa, 2, \theta, 2)$ fails for every $\theta < \kappa$.
We also show that,
if $\lambda$ is a singular limit of strongly compact cardinals, then
$\U(\lambda^+, 2, \theta, \cf(\lambda)^+)$ fails for all
$\theta \in \reg(\lambda) \setminus \{\cf(\lambda)\}$.

We next turn to applications of $\U(\ldots)$. Our primary result regarding the
infinite productivity of strong chain conditions is as follows and answers a
question raised at the end of the Introduction of \cite{lh_lucke}

\begin{thmd}
  Suppose that $\theta, \chi < \kappa$ are infinite regular cardinals,
  $\kappa$ is $({<}\chi)$-inaccessible, and $\U(\kappa, \kappa, \theta, \chi)$
  holds. Then there exists a $\chi$-directed closed poset $\mathbb{P}$
  such that $\mathbb{P}^\tau$ is $\kappa$-Knaster for all $\tau < \min(\{\chi, \theta\})$,
  but $\mathbb{P}^\theta$ is not $\kappa$-cc.
  In particular, for every infinite successor cardinal $\kappa$,
  there exists a $\kappa$-Knaster poset $\mathbb{P}$ whose $\omega^{\mathrm{th}}$
  power is not $\kappa$-cc.
\end{thmd}

\begin{proof}
  This follows from Lemma~\ref{knasterlemma} and Corollary~\ref{omegacolors}.
\end{proof}

We also present an unpublished result of Inamdar \cite{tanmay_note} indicating a fundamental
limitation to generalizations of Martin's Axiom to higher cardinals. Implicit in
Inamdar's proof was a use of $\U(\lambda^+, \lambda^+, \omega, \lambda)$ for regular,
uncountable $\lambda$ that helped motivate some of the work in this paper.

\subsection{Structure of the paper and its sequels}

In Section~\ref{prelim_section}, we present some basic facts about $\U(\ldots)$
to lay the framework for further analysis. In particular, we consider some
elementary implications and non-implications that exist between various instances
of $\U(\ldots)$, prove some limitations placed on $\U(\dots)$ by the existence of
large cardinals, and we discuss some properties of trees derived from witnesses to
$\U(\ldots)$.

In Section~\ref{application_section}, we present the primary applications of
the paper. Subsection~\ref{productivity_subsection} contains our results regarding
$\U(\ldots)$ and the infinite productivity of strengthenings of the $\kappa$-chain
condition. In Subsection~\ref{forcing_axioms_subsection}, we present the
aforementioned result of Inamdar concerning generalizations of Martin's Axiom
to higher cardinals.

In Section~\ref{closed_section}, we prove our results regarding circumstances
under which instances of $\U(\ldots)$ necessarily hold. In all cases, our
proofs will in fact yield witnesses to $\U(\ldots)$ with certain closure
properties that make them better-behaved. In Subsection~\ref{closed_subsection},
we present some basic facts about these closed colorings. In Subsection~\ref{walks_subsection},
we review the necessary background concerning \emph{walks on ordinals}, which
provide our main tool for constructing witnesses to $\U(\ldots)$.
Subsection~\ref{first_construction_subsection} contains our first construction of
such witnesses, in particular yielding the fact that all possible instances of
$\U(\ldots)$ hold at successors of regular cardinals. Subsections \ref{singular_subsection}
and \ref{inaccessible_subsection} contain further constructions at successors of
singular cardinals and inaccessible cardinals, respectively.

\medskip

As the title of the paper suggests, it is the first paper in a series. In Part~II, we introduce a new cardinal
invariant for regular uncountable cardinals, the \emph{$C$-sequence number}, which
is intimately connected to the fourth parameter of $\U(\ldots)$ as well as to various
square principles. Considerations of the $C$-sequence number will allow us
to obtain additional results regarding the existence
of closed witnesses to $\U(\ldots)$. In Part~III, we study the existence of \emph{subadditive} witnesses to
$\U(\ldots)$ and discuss applications of such subadditive witnesses to the
infinite productivity of further strengthenings of the $\kappa$-chain condition
and to topological matters, such as the question as to the tightness of the
square of the sequential fan. The techniques of parts II and III will also allow us to prove
independence results separating certain instances of $\U(\ldots)$ at inaccessible
cardinals and successors of singular cardinals. Where relevant, we will make
reference to results in Parts II and III that will provide further context for the results
in this paper, though no knowledge of these papers is necessary for any of
the results contained here.

\subsection{Notation and conventions}

Throughout the paper, $\kappa$ denotes a regular uncountable cardinal, and $\chi,\theta,$ and $\mu$
denote cardinals $\le\kappa$.
We say that $\kappa$ is \emph{$({<}\chi)$-inaccessible} iff, for all $\lambda<\kappa$ and
$\nu<\chi$, $\lambda^{\nu}<\kappa$. $\reg$ denotes the class of infinite regular
cardinals, and $\reg(\kappa)$ denotes $\reg \cap \kappa$. $E^\kappa_\chi$
denotes the set $\{\alpha < \kappa \mid \cf(\alpha) = \chi\}$, and
$E^\kappa_{\geq \chi}$, $E^\kappa_{>\chi}$, $E^\kappa_{\neq\chi}$, etc.\ are defined analogously.
For a subset $S\s\kappa$, we let $\Tr(S):=\{\alpha\in E^\kappa_{>\omega}\mid S\cap\alpha\text{ is stationary in }\alpha\}$;
We say that $S$ is \emph{non-reflecting} (resp.\ \emph{non-reflecting at inaccessibles}) iff $\Tr(S)$ is empty (resp.\ contains no inaccessible cardinals).
The principle $\refl({<}\theta,S)$ asserts that for every family $\mathcal S$ consisting of
less than $\theta$-many stationary subsets of $S$, the set $\bigcap_{S\in\mathcal S}\Tr(S)$ is nonempty.

For an ideal $\mathcal I$ on $\kappa$, we write $\mathcal I^+:=\mathcal P(\kappa)\setminus\mathcal I$,
and $\mathcal I^*:=\{ \kappa\setminus X\mid X\in\mathcal I\}$.

For the definitions of the principles $\square(\kappa,\sq_\sigma)$, $\square(\kappa)$ and $\square_\lambda^*$,
see \cite{paper29}, Definition~1.16, and the discussion following it.
For the definitions of the principles $\diamondsuit(S)$, $\clubsuit^-(S)$ and $\clubsuit(S)$,
see \cite{paper_s01},  Definitions 1.1, 1.18 and 2.15, respectively.

For a set of ordinals $a$, we write $\ssup(a) := \sup\{\alpha + 1 \mid
\alpha \in a\}$, $\acc^+(a) := \{\alpha < \ssup(a) \mid \sup(a \cap \alpha) = \alpha > 0\}$,
$\acc(a) := a \cap \acc^+(a)$, $\nacc(a) := a \setminus \acc(a)$,
and $\cl(a):= a\cup\acc^+(a)$.
For sets of ordinals, $a$ and $b$, we write $a < b$ if, for all $\alpha \in a$
and all $\beta \in b$, we have $\alpha < \beta$.
For a set of ordinals $a$ and an ordinal $\beta$, we write
$a < \beta$ instead of $a < \{\beta\}$ and $\beta < a$ instead of $\{\beta\} < a$.

For any set $\mathcal A$, we write
$[\mathcal A]^\chi:=\{ \mathcal B\s\mathcal A\mid |\mathcal B|=\chi\}$ and
$[\mathcal A]^{<\chi}:=\{\mathcal B\s\mathcal A\mid |\mathcal B|<\chi\}$.
In particular, $[\mathcal{A}]^2$ consists of all unordered pairs from $\mathcal{A}$.
In some scenarios, we will also be interested in ordered pairs from $\mathcal{A}$.
In particular, if $\mathcal{A}$ is either an ordinal or a collection of sets
of ordinals, then we will abuse notation and write $(a,b) \in [\mathcal{A}]^2$
to mean $\{a,b\} \in [\mathcal{A}]^2$ and $a < b$.

\section{Preliminary results} \label{prelim_section}

In this section, we present some basic results regarding $\U(\kappa, \mu, \theta, \chi)$.
We begin by cataloging some implications that exist between various incarnations
of the coloring principles under consideration. The following proposition is immediate.

\begin{prop}\label{obvious}
  \begin{enumerate}
    \item $\U(\kappa,\kappa,\kappa,\kappa)$ holds.
    \item \label{singular} If $\cf(\theta') = \cf(\theta)$, then $\U(\kappa, \mu, \theta, \chi)$
      holds iff $\U(\kappa, \mu, \theta', \chi)$ holds.
    \item For all $\mu'\le\mu$ and $\chi'\le\chi$, $\U(\kappa,\mu,\theta,\chi)$
      entails $\U(\kappa,\mu',\theta,\chi')$.
    \item \label{singular2} If $\chi$ is a limit cardinal and $c:[\kappa]^2 \rightarrow\theta$
      witnesses $\U(\kappa,\mu,\theta,\chi')$ for all $\chi'<\chi$, then $c$ witnesses
      $\U(\kappa,\mu,\theta,\chi)$, as well.
    \item If $\U(\kappa,2,\theta,\chi)$ and $\kappa\rightarrow(\kappa,\mu)^2$ both hold,
      then $\U(\kappa,\mu,\theta,\chi)$ holds, as well. In particular,
      $\U(\kappa, 2, \theta, \chi)$ entails $\U(\kappa, \omega, \theta, \chi)$.
    \item If $\chi<\kappa$, then $\U(\kappa,\mu,\theta,\chi)$ holds iff there exists
      a coloring $c:[\kappa]^2\rightarrow\theta$ such that,
      for every family $\mathcal A\s[\kappa]^{<\chi}$ consisting of $\kappa$-many pairwise disjoint sets,
      and for every $i<\theta$, there exists $\mathcal B\in[\mathcal A]^\mu$ such that $\min(c[a\times b])>i$
      for all $(a,b)\in[\mathcal{B}]^2$.\qed
    \end{enumerate}
\end{prop}
\begin{remark}For Clause~(5), recall that $\kappa\rightarrow(\kappa,\mu)^2$ stands for the assertion that for every
coloring $c\colon [\kappa]^2\rightarrow2$, either there exists $A\in[\kappa]^\kappa$ such that
$c``[A]^2=\{0\}$, or there exists $B\in[\kappa]^\mu$ such that $c``[B]^2=\{1\}$.
By a classic theorem of Dushnik and Miller, $\kappa\rightarrow(\kappa,\omega)^2$
holds for every infinite cardinal $\kappa$.
\end{remark}

Because of Clauses (1) and (2) of the preceding Proposition, we shall focus throughout
on the case in which $\theta\in\reg(\kappa)$. We next note that instances of $\pr_1(\ldots)$ easily yield
instances of $\U(\ldots)$.

\begin{lemma}
  Suppose that $\pr_1(\kappa,\kappa,\theta,\chi)$ holds.
  Then $\U(\kappa,2,\theta',\chi)$ holds for all $\theta'\le\theta$.
\end{lemma}

\begin{proof}
  Let $c:[\kappa]^2 \rightarrow \theta$ witness $\pr_1(\kappa, \kappa, \theta, \chi)$,
  and fix $\theta' \leq \theta$. Define $c':[\kappa]^2 \rightarrow \theta'$ by
  setting, for all $(\alpha, \beta) \in [\kappa]^2$,
  \[
    c'(\alpha, \beta) :=
    \begin{cases}
      c(\alpha, \beta) &\text{if }c(\alpha, \beta) < \theta';\\
      0 & \text{if } c(\alpha, \beta) \geq \theta'.
    \end{cases}
  \]
  It is easily verified that $c'$ witnesses $\U(\kappa, 2, \theta', \chi)$.
\end{proof}

We next present some lemmas about increasing the second and fourth parameters
in instances of $\U(\ldots)$.

\begin{lemma}
  Suppose that $\U(\kappa,\mu,\theta,\chi)$ holds for all $\chi\in X$.
  If $\cf(\sup(X))<\cf(\theta)$, then $\U(\kappa,\mu,\theta,\sup(X))$ holds, as well.
\end{lemma}

\begin{proof}
  Since $\cf(\sup(X)) < \cf(\theta)$, we may assume, by thinning out $X$ if necessary,
  that $|X| < \cf(\theta)$. For each $\chi \in X$, let $c_\chi:[\kappa]^2 \rightarrow \theta$
  witness $\U(\kappa, \mu, \theta, \chi)$. Define $c:[\kappa]^2 \rightarrow \theta$
  by letting $c(\alpha, \beta) := \sup\{c_\chi(\alpha, \beta) \mid \chi \in X\}$
  for all $(\alpha, \beta) \in [\kappa]^2$. It is easily verified that $c$ witnesses
  $\U(\kappa, \mu, \theta, \sup(X))$.
\end{proof}

\begin{lemma}\label{lemma24}
  Suppose that $\theta\le\lambda$ are infinite cardinals. Then $\U(\lambda^+,2,\theta,2)$ holds
  iff $\U(\lambda^+,\omega,\theta,\cf(\lambda))$ holds.
\end{lemma}

\begin{proof}
  Clearly, only the forward implication needs an argument. By Corollary~\ref{l23} below,
  if $\lambda$ is regular, then in fact $\U(\lambda^+, \lambda^+, \theta, \lambda)$
  holds, so we may assume that $\lambda$ is singular. Fix a function
  $c:[\lambda^+]^2\rightarrow\theta$ witnessing $\U(\lambda^+,2,\theta,2)$.
  By~\cite[Theorem~3.1]{paper13}, we may pick a function $t:[\lambda^+]^2
  \rightarrow[\lambda^+]^2$ with the property that, for every family
  $\mathcal A\s[\lambda^+]^{<\cf(\lambda)}$ consisting of $\lambda^+$-many
  pairwise disjoints sets, there exists a stationary $S\s\lambda^+$ such that,
  for all $(\alpha,\beta)\in[S]^2$, there exists $(a,b)\in[\mathcal A]^2$
  such that $t[a\times b]=(\alpha,\beta)$. It is easy to see that $c\circ t$ witnesses
  $\U(\lambda^+,2,\theta,\cf(\lambda))$, and hence, by Proposition~\ref{obvious}(5),
  also $\U(\lambda^+,\omega,\theta,\cf(\lambda))$.
\end{proof}

\begin{lemma}\label{chi3}
  If a coloring $c:[\kappa]^2 \rightarrow \theta$ witnesses $\U(\kappa,\kappa,\theta,3)$, then it also witnesses
  $\U(\kappa,\kappa,\theta,\omega)$.
\end{lemma}

\begin{proof}
  Suppose that $c:[\kappa]^2\rightarrow\theta$ witnesses $\U(\kappa,\kappa,\theta,3)$.
  We prove by induction on $n\ge2$ that $c$ witnesses $\U(\kappa,\kappa,\theta,n+1)$.
  To this end, fix an integer $n \geq 2$, and suppose that $c$ witnesses
  $\U(\kappa, \kappa, \theta, n)$. To show that $c$ witnesses $\U(\kappa, \kappa, \theta,
  n+1)$, it suffices to show that, for every $\mathcal A\s[\kappa]^n$ consisting of
  $\kappa$-many pairwise disjoint sets, and every $i<\theta$, there exists $\mathcal B\in[\mathcal A]^\kappa$
  such that $\min(c[a\times b])>i$ for all $(a,b)\in[\mathcal B]^2$.

  Fix such an $\mathcal{A}$ and $i$. Suppose that $\mathcal{A}$ is injectively enumerated
  as $\{a_\alpha \mid \alpha < \kappa\}$ and that, for all $\alpha < \kappa$,
  $a_\alpha$ is enumerated as $\{a_{\alpha, j} \mid j < n\}$. Let
  $\{ p_l\mid l<{n\cdot(n-1)\over 2}\}$ be an injective enumeration of $\{ \{j,j'\} \mid j<j'<n\}$.
  As $c$ witnesses $\U(\kappa,\kappa,\theta,3)$,
  we may recursively find a $\s$-decreasing chain $\langle T_l \mid l<{n\cdot(n-1)\over 2}\rangle$
  of elements of $[\kappa]^\kappa$ such that for each $l$, for all distinct $\alpha,\beta$ in $T_l$,
  letting $a:=\{ a_{\alpha,j}\mid j\in p_l\}$ and $b:=\{ a_{\beta,j}\mid j\in p_l\}$,
  if $a < b$, then $\min(c[a\times b]) > i$.
  Now, let $\mathcal B:=\{ a_\alpha \mid \alpha\in\bigcap_{l<{n\cdot(n-1)\over 2}}T_l\}$.
  Clearly, for every $(a,b)\in[\mathcal B]^2$, we have $\min(c[a\times b])>i$.
\end{proof}

The preceding lemma is optimal in the following sense.

\begin{lemma}
  Suppose that for some cardinal $\lambda$, $(\lambda^{<\theta})^+\le\kappa\le\lambda^\theta$.
  Then there exists a coloring witnessing $\U(\kappa,\kappa,\theta,2)$
  that fails to witness $\U(\kappa,2,\theta,3)$.
\end{lemma}

\begin{proof}
  Fix an injective enumeration $\{f_\alpha\mid\alpha<\kappa\}$ of some subset of
  ${}^\theta\lambda$. Define a coloring $c:[\kappa]^2\rightarrow\theta$ by letting,
  for all $\alpha<\beta<\kappa$,
  $$c(\alpha,\beta):=\min\{i<\theta\mid f_\alpha(i)\neq f_\beta(i)\}.$$

  To see that $c$ witnesses $\U(\kappa,\kappa,\theta,2)$,
  fix an arbitrary $\mathcal A\s\kappa$ of size $\kappa$ and a color $i<\theta$.
  As $\cf(\kappa)=\kappa>|{}^{(i+1)}\lambda|$,
  we may pick $h:i+1\rightarrow\lambda$ for which $\mathcal A_h:=\{ \alpha\in A\mid h\s f_\alpha\}$ has size
  $\kappa$. Clearly, $c(\alpha, \beta)>i$ for all $(\alpha,\beta)\in[\mathcal A_h]^2$.

  To see that $c$ does not witness $\U(\kappa,2,\theta,3)$, define $g:\acc(\kappa)\rightarrow\theta$
  by letting $g(\alpha):=c(\alpha,\alpha+1)$ for all $\alpha\in\acc(\kappa)$.
  Pick $A\in[\acc(\kappa)]^\kappa$ on which $g$ is constant, with value, say, $i$.
  Clearly, for every $(\alpha,\beta)\in[A]^2$, if $c(\alpha,\beta)>i$,
  then $c(\alpha,\beta+1)=c(\beta,\beta+1)=i$.
  It follows that $\mathcal A:=\{ \{\alpha,\alpha+1\}\mid \alpha\in A\}$
  is a family consisting of $\kappa$-many pairwise disjoint sets,
  and for all $(a,b)\in\mathcal A$, we have $\min(c[a\times b])\le i$.
\end{proof}

\subsection{Associated trees}
In this subsection, we begin to investigate trees derived from colorings, particularly
those witnessing instances of $\U(\ldots)$. For any coloring $c:[\kappa]^2\rightarrow \kappa$
and any ordinal $\gamma<\kappa$, we denote by $c({\cdot},\gamma)$ the unique function
from $\gamma$ to $\kappa$ satisfying $c({\cdot},\gamma)(\alpha)=c(\alpha,\gamma)$
for all $\alpha<\gamma$. Then, the \emph{tree associated to $c$} is
$$
  \mathcal T(c):=\{c({\cdot},\gamma)\restriction\beta\mid \beta\le\gamma<\kappa\}.
$$

We begin by proving that, if $c$ witnesses certain mild instances of $\U(\ldots)$,
then $\mathcal{T}(c)$ cannot admit a cofinal branch.

\begin{prop}\label{aronszajn}
  Suppose that $\theta < \kappa$ and $c:[\kappa]^2\rightarrow\theta$ witnesses $\U(\kappa,2,\theta,2)$.
  Then $\mathcal T(c)$ admits no cofinal branch.
\end{prop}

\begin{proof}
  Suppose not, and fix $b:\kappa\rightarrow\theta$ such that $\{b\restriction\beta\mid \beta<\kappa\}\s\mathcal T(c)$.
  This means that for every ordinal $\beta<\kappa$, there exists some $\gamma\in[\beta,\kappa)$,
  such that $b\restriction\beta\s c(\cdot,\gamma)$. Recursively construct a strictly
  increasing function  $f:\kappa\rightarrow\kappa$ such that
  $b\restriction\ssup(f[\alpha])\s c(\cdot,f(\alpha))$ for all $\alpha<\kappa$.
  Write $A:=\im(f)$. In particular, $c(\alpha,\gamma)=b(\alpha)$ for all $(\alpha,\gamma)\in[A]^2$.

  As $\theta<\kappa$, let us fix $B\in[A]^\kappa$ on which $\alpha\mapsto b(\alpha)$ is constant with value, say, $i$.
  Then $\sup(c``[B]^2)\le i$,
  contradicting the fact that $c$ witnesses $\U(\kappa, 2, \theta, 2)$.
\end{proof}
\begin{remark} It is natural to ask whether the tree associated to a witness to $\U(\cdots)$
must be special or must be nonspecial. We shall address this question in Subsection~\ref{first_construction_subsection}.
\end{remark}

Recall that a \emph{$\theta$-ascending path} (resp.\ \emph{$\theta$-ascent path}) through a $\kappa$-tree $(T,<_T)$
is a sequence $\langle f_\alpha\mid \alpha<\kappa\rangle$ such that the following two conditions hold:
\begin{itemize}
\item for all $\alpha<\kappa$, $f_\alpha$ is a function from $\theta$ to the $\alpha^{\mathrm{th}}$ level of $(T,<_T)$;
\item for all $\alpha<\beta<\kappa$, there are $j,j'<\theta$ such that $f_\alpha(j)<_T f_\beta(j')$ (resp.\
$f_\alpha(j)<_T f_\beta(j)$ for a tail of $j<\theta$).
\end{itemize}

In Part~III, it is proved that, if $\theta<\kappa$ and there exists a $\kappa$-tree
admitting a $\theta$-ascent path but no $\theta'$-ascent path for $\theta'<\theta$,
then $\U(\kappa,2,\theta,\theta)$ holds.
We now generalize Proposition~\ref{aronszajn} and deal with the converse of the result from Part~III.

\begin{lemma}\label{lemma610}
  Suppose that $\theta \in \reg(\kappa)$, $\chi < \kappa$, and $c:[\kappa]^2\rightarrow\theta$
  witnesses $\U(\kappa,2,\theta,\chi)$. For every infinite cardinal $\theta'<\chi$,
  \begin{enumerate}
    \item if $\theta'<\theta$, then $\mathcal T(c)$ admits no $\theta'$-ascending path;
    \item if $\cf(\theta')\neq\theta$, then $\mathcal T(c)$ admits no $\theta'$-ascent path.
  \end{enumerate}
\end{lemma}

\begin{proof}
  For each $\alpha<\kappa$, set $T_\alpha:=\mathcal T(c)\cap{}^\alpha\theta$.
  Suppose that $\theta'$ is an infinite cardinal less than $\chi$ and that $\langle f_\alpha:
  \theta'\rightarrow T_\alpha\mid \alpha<\kappa\rangle$ is a given sequence of functions.
  For each $\alpha<\kappa$, fix $a_\alpha\in[\kappa]^{<\chi}$ such that
  $$
    \{ c(\cdot,\alpha)\}\cup\{ f_{\alpha+1}(j)\mid j<\theta'\}=\{ c(\cdot,\beta)
    \restriction(\alpha+1)\mid \beta\in a_\alpha\}.
  $$
  As $\min(a_\alpha)=\alpha$, we may pick $A\in[\kappa]^\kappa$ for which
  $\{ a_\alpha\mid \alpha\in A\}$ are pairwise disjoint.

  (1) Suppose that $\theta'<\theta$. For each $\alpha\in A$, let $i_\alpha:=
  \sup\{ f_{\alpha+1}(j)(\alpha)\mid j<\theta'\}$. Fix $B\in[A]^\kappa$ for which
  $\{i_\alpha\mid\alpha\in B\}$ is a singleton, say $\{i\}$. Pick
  $(\alpha,\alpha')\in[B]^2$ with $a_\alpha<a_{\alpha'}$ such that $\min(c[a_\alpha\times a_{\alpha'}])>i$.

  Towards a contradiction, suppose that there exist $j,j'<\theta'$ such that
  $f_{\alpha+1}(j) \s f_{\alpha'+1}(j')$. Pick $\beta\in a_{\alpha'}$ such that
  $f_{\alpha'+1}(j')\s c(\cdot,\beta)$. Then $c(\alpha,\beta)=
  f_{\alpha'+1}(j')(\alpha)=f_{\alpha+1}(j)(\alpha)\le i$, contradicting the fact that
  $(\alpha,\beta)\in a_\alpha\times a_{\alpha'}$.

  (2)  Suppose that $\cf(\theta')\neq\theta$. If $\cf(\theta')<\theta$ and
  $\mathcal T(c)$ admits a $\theta'$-ascent path,
  then it also admits a $\cf(\theta')$-ascent path, and hence a $\cf(\theta')$-ascending
  path. This case is therefore covered by Clause~(1).

  Next, suppose that $\cf(\theta')>\theta$. For each $\alpha<\kappa$,
  pick $i_\alpha<\theta$ such that $\sup\{j<\theta'\mid f_{\alpha+1}(j)(\alpha)=i_\alpha\}=\theta'$.
  Fix $B\in[A]^\kappa$ for which $\{i_\alpha\mid\alpha\in B\}$ is a singleton, say $\{i\}$.
  Pick  $(\alpha,\alpha') \in [B]^2$ with $a_\alpha < a_{\alpha'}$
  such that $\min(c[a_\alpha\times a_{\alpha'}])>i$.

  Towards a contradiction, suppose that $f_{\alpha+1}(j) \s f_{\alpha' +  1}(j)$
  for all sufficiently large $j < \theta'$, and use this to find a $j$ for which
  $f_{\alpha+1}(j) \s f_{\alpha'+1}(j)$ and $f_{\alpha+1}(j)(\alpha)=i_\alpha$.
  Pick $\beta\in a_{\alpha'}$ such that $f_{\alpha'+1}(j)\s c(\cdot,\beta)$. Then $c(\alpha,\beta)=
  f_{\alpha'+1}(j)(\alpha)=f_{\alpha+1}(j)(\alpha)=i_\alpha=i$, contradicting the fact that
  $(\alpha,\beta)\in a_\alpha\times a_{\alpha'}$.
\end{proof}

\begin{remark}
  In \cite{lh_lucke}, Lambie-Hanson and L\"{u}cke prove that, if
  $\kappa$ is a weakly compact cardinal and $\theta'\in\reg(\kappa)$,
  then, in some cofinality-preserving forcing extension, $\kappa$ remains strongly inaccessible
  and every $\kappa$-tree has a $\theta'$-ascent path.
  It follows that, in their model, $\U(\kappa, 2, \theta, (\theta')^+)$ fails for all $\theta\in\reg(\kappa)\setminus\{\theta'\}$.
  In Part~II, we shall carry out a further analysis of this model, proving that it satisfies $\U(\kappa, \kappa, \theta', \kappa)$.
  In particular, the special case $\theta'=\omega$ will yield a model in which
  $\Pr_1(\kappa,\kappa,2,2)$ holds and $\pr_1(\kappa,\kappa,2,\omega_1)$ fails,
  thus showing that \cite[Conjecture~2]{paper18} is the most one can hope for.
\end{remark}

\begin{cor}\label{cor611} Suppose that $\kappa$ is a strongly inaccessible cardinal and $\U(\kappa,2,\omega_1,\omega_1)$ holds.
Then there exists a $\kappa$-Aronszajn tree with no $\omega$-ascending path.\qed
\end{cor}

\subsection{Large cardinals}

In this subsection, we indicate how large cardinals can imply nontrivial failures
of $\U(\ldots)$ at inaccessible cardinals and successors of singular cardinals.
First, recall that a cardinal $\kappa$ is \emph{weakly compact} if
it is strongly inaccessible and there are no $\kappa$-Aronszajn trees.
The following fact is now an immediate consequence of Proposition~\ref{aronszajn}.
\begin{fact}\label{fact29}
  If $\kappa$ is weakly compact, then $\U(\kappa, 2, \theta, 2)$ fails for all
  $\theta\in\reg(\kappa)$.
\end{fact}

To obtain a similar result at successors of singular cardinals, we employ
strongly compact cardinals. Recall that a cardinal $\nu$ is \emph{strongly
compact} if it is uncountable and every $\nu$-complete filter can be extended
to a $\nu$-complete ultrafilter.

\begin{thm}\label{t310}
  Suppose that $\lambda$ is a singular limit of strongly compact cardinals.
  Then $\U(\lambda^+, 2, \theta, \cf(\lambda)^+)$ fails for all
  $\theta\in\reg(\lambda)\setminus\{\cf(\lambda)\}$.
\end{thm}

\begin{proof}
  Fix an arbitrary $\theta\in\reg(\lambda)\setminus\{\cf(\lambda)\}$ and a coloring
  $c : [\lambda^+]^2 \rightarrow \theta$. In order to show that $c$ does not witness
  $\U(\lambda^+, 2, \theta, \cf(\lambda)^+)$, we will find an $i < \theta$ and a family $\mathcal{A}
  \subseteq [\lambda^+]^{\leq \cf(\lambda)}$ consisting of $\lambda^+$-many pairwise disjoint sets
  such that, for all $(a,b) \in [\mathcal{A}]^2$, we have $\min(c[a \times b]) \leq i$.

  Let $\langle \lambda_j \mid j < \cf(\lambda) \rangle$ be an increasing sequence of
  strongly compact cardinals that is cofinal in $\lambda$, with $\lambda_0>\theta$.
  For a fixed $j < \cf(\lambda)$, use the strong compactness of $\lambda_j$ to pick a
  uniform, $\lambda_j$-complete ultrafilter $U_j$ on $\lambda^+$. Then,
  for each $\alpha < \lambda^+$, use the $\lambda_j$-completeness
  of $U_j$ to find an $i^j_\alpha < \theta$ and an $X^j_\alpha \in
  U_j$ such that for all $\beta \in
  X^j_\alpha$, we have $\alpha<\beta$ and $c(\alpha, \beta) = i^j_\alpha$.
  Then, again use the completeness of $U_j$ to
  find an $i^j < \theta$ and a $Y^j \in U_j$ such that, for all
  $\alpha \in Y^j$, we have $i^j_\alpha = i^j$.

  Now, since $\theta$
  is regular and $\theta \neq \cf(\lambda)$, we may find an $i < \theta$ and an
  unbounded $J \subseteq \cf(\lambda)$ such that, for all $j \in J$, we have $i^j \leq i$.
  For every nonzero $\delta<\lambda^+$, fix a sequence $\langle Z_\delta^j \mid j \in J\rangle$ such that
  \begin{itemize}
    \item $\bigcup_{j \in J} Z_\delta^j = \delta$;
    \item for all $j \in J$, $0 < |Z_\delta^j| < \lambda_j$.
  \end{itemize}

  We now construct our family $\mathcal{A} := \{a_\gamma \mid \gamma < \lambda^+\}$
  by recursion on $\gamma < \lambda^+$. We will arrange so that, for all $\gamma <
  \delta < \lambda^+$,
  \begin{itemize}
    \item for all $j \in J$, we have $a_\gamma \cap Y^j \neq \emptyset$;
    \item $a_\gamma < a_\delta$;
    \item $\min(c[a_\gamma \times a_\delta]) \leq i$.
  \end{itemize}
  This will clearly suffice to prove the theorem.

  Begin by letting $a_0 := \{\min(Y^j) \mid j \in J\}$. Next, suppose that
  $\delta\in\lambda^+\setminus\{0\}$ and we have already constructed $\{a_\gamma \mid \gamma < \delta\}$.
  Let $\epsilon_\delta := \sup(\bigcup_{\gamma < \delta}a_\gamma)$.
  For each $j \in J$ and each $\gamma \in Z_\delta^j$, fix $\alpha^j_\gamma \in a_\gamma \cap Y^j$.
  For each $j \in J$, use the $\lambda_j$-completeness
  of $U_j$ to find $\beta^j \in Y^j \cap \bigcap_{\gamma \in Z_\delta^j} X^j_{\alpha^j_\gamma}$
  with $\beta^j > \epsilon_\delta$. Let $a_\delta := \{\beta^j \mid j \in J\}$.

  It remains to check that we have maintained the recursion hypotheses. We clearly
  have both $a_\delta \cap Y^j \neq \emptyset$ for all $j \in J$, and also
  $a_\gamma < a_\delta$ for all $\gamma < \delta$. To see that
  $\min(c[a_\gamma \times a_\delta]) \leq i$ for all $\gamma < \delta$, fix
  such a $\gamma$ and fix $j \in J$ such that $\gamma \in Z_\delta^j$.
  Then $\alpha^j_\gamma \in a_\gamma \cap Y^j$ and $\beta^j \in a_\gamma
  \cap X^j_{\alpha^j_\gamma}$, so $c(\alpha^j_\gamma, \beta^j)
  = i^j \leq i$, and we are done.
\end{proof}

\begin{remark}
  Let us note here two ways in which the preceding result is optimal. First,
  by Corollary~\ref{predecessor} below, $\U(\lambda^+, \lambda^+, \cf(\lambda),
  \lambda)$ holds for every singular cardinal $\lambda$, so the requirement
  ``$\theta \neq \cf(\lambda)$" cannot be waived. Second, recall that $\sch$
  holds above a strongly compact cardinal. In particular, in the setting of
  the preceding result, we have $2^\lambda = \lambda^+$. It then follows from
  Theorem~\ref{ch} below that $\U(\lambda^+, \lambda^+, \theta, \cf(\lambda))$
  holds for all $\theta < \lambda$, so the fourth parameter cannot be reduced
  from $\cf(\lambda)^+$.
\end{remark}

In Part~II, we will force over models with large cardinals to obtain a finer separation
between instances of $\U(\ldots)$. In particular, we will obtain the following consistency
results.
\begin{itemize}
  \item We will force with a cofinality-preserving forcing notion over a model
    in which $\kappa$ is weakly compact and $\theta < \kappa$ is regular to obtain
    a model in which $\U(\kappa, \kappa, \theta, \kappa)$ holds but
    $\U(\kappa, 2, \theta', \theta^+)$ fails for every $\theta' \in
    \reg(\kappa) \setminus \{\theta\}$.
  \item We will force with a cofinality-preserving forcing notion over a model
    in which $\lambda$ is a singular limit of supercompact cardinals and
    $\theta$ is a regular cardinal with $\cf(\lambda) \leq \theta < \lambda$
    to obtain a model in which $\U(\lambda^+, \lambda^+, \theta, \lambda)$
    holds but $\U(\lambda^+, 2, \theta', \theta^+)$ fails for all
    $\theta' \in \reg(\lambda) \setminus \{\cf(\lambda), \theta\}$.
\end{itemize}

\section{Strong chain conditions and forcing axioms} \label{application_section}

The work in this paper arose in part from questions regarding the infinite productivity
of chain conditions and possible generalizations of Martin's Axiom to higher cardinals.
In this section, we present these questions and indicate how the property
$\U(\kappa, \mu, \theta, \chi)$ comes to bear on them.

\subsection{Infinite productivity of strong chain conditions} \label{productivity_subsection}

We start this section by recalling some relevant properties of posets, starting
with closure properties.

\begin{defn}
  Let $\mathbb{P}$ be a poset and $\lambda$ be a regular, uncountable cardinal.
  \begin{enumerate}
    \item $\mathbb{P}$ is \emph{well-met} if, whenever $p,q \in \mathbb{P}$
      are compatible, they have a greatest lower bound in $\mathbb{P}$.
    \item $\mathbb{P}$ is \emph{$\lambda$-closed (resp.\ $\lambda$-directed
    closed) with greatest lower bounds} if, whenever $\tau < \lambda$
    and $\langle q_\eta \mid \eta < \tau \rangle$ is $\leq_{\mathbb{P}}$-decreasing
    (resp.\ $\leq_{\mathbb{P}}$-directed), it has a greatest lower bound in $\mathbb{P}$.
  \end{enumerate}
\end{defn}

We next recall some strengthenings of the $\kappa$-chain condition.

\begin{defn}
  Let $\mathbb{P}$ be a poset.
  \begin{enumerate}
    \item A subset $B \subseteq \mathbb{P}$ is \emph{linked} if it consists of
      pairwise compatible conditions. $B$ is \emph{centered} if every finite
      subset of $B$ has a lower bound in $\mathbb{P}$.
    \item $\mathbb{P}$ is \emph{$\kappa$-Knaster} if, whenever $A \in [\mathbb{P}]^\kappa$,
      there is $B \in [A]^\kappa$ that is linked.
    \item $\mathbb{P}$ has \emph{precaliber $\kappa$} if, whenever $A \in
      [\mathbb{P}]^\kappa$, there is $B \in [A]^\kappa$ that is centered.
    \item For a cardinal $\lambda$, $\mathbb{P}$ is \emph{$\lambda$-centered}
      (resp.\ \emph{$\lambda$}-linked) if
      there is a collection of $\lambda$-many centered (resp.\ linked) subsets of $\mathbb{P}$
      that covers $\mathbb{P}$. Note that, if $\lambda < \kappa$ and $\mathbb{P}$
      is $\lambda$-centered or $\lambda$-linked, then $\mathbb{P}$ has the $\kappa$-cc.
  \end{enumerate}
\end{defn}

One nice feature of these strong chain conditions is the fact that they are
\emph{productive}, i.e., if $\mathbb{P}$ and $\mathbb{Q}$ are $\kappa$-Knaster
(or have precaliber $\kappa$ or are $\lambda$-linked or $\lambda$-centered),
then $\mathbb{P} \times \mathbb{Q}$ is $\kappa$-Knaster (or has precaliber
$\kappa$ or is $\lambda$-linked or $\lambda$-centered, respectively). This is
in contrast to the $\kappa$-cc, which is not in general productive.

It is natural to investigate the extent to which these chain conditions can
be more than finitely productive. Note that, if $\theta < \kappa$ and
$\mathbb{P}_\eta$ has the $\kappa$-cc (is $\kappa$-Knaster or has precaliber $\kappa$, resp.)
for all $\eta < \theta$, then the lottery sum $\mathbb{P} := \bigoplus_{\eta < \theta}
\mathbb{P}_\eta$ also has the $\kappa$-cc (is $\kappa$-Knaster or has precaliber $\kappa$, resp.),
so questions about the productivity of these conditions reduce to questions about
powers of forcing posets. In particular, for regular cardinals $\theta
< \kappa$, we are interested in the following question: If $\mathbb{P}$ is a poset
such that $\mathbb{P}^\tau$ is $\kappa$-Knaster (resp.\ has precaliber $\kappa$)
for all $\tau < \theta$, does it follow that $\mathbb{P}^\theta$ is
$\kappa$-Knaster (resp.\ has precaliber $\kappa$). Note that, if $\theta = \aleph_0$,
this is simply asking whether being $\kappa$-Knaster (resp.\ having precaliber $\kappa$)
is countably productive. One can also ask for weaker conclusions, e.g., if
$\mathbb{P}^\tau$ is $\kappa$-Knaster for all $\tau < \theta$, does it follow
that $\mathbb{P}^\theta$ has the $\kappa$-cc?

If $\kappa$ is a weakly compact cardinal, then a poset has the $\kappa$-cc if
and only if it is $\kappa$-Knaster, and both of these properties are $\theta$-productive
for all $\theta < \kappa$ (i.e., if $\mathbb{P}$ is $\kappa$-Knaster, then
$\mathbb{P}^\theta$ is also $\kappa$-Knaster). In \cite{MR3620068}, Cox and L\"{u}cke
show that, relative to the consistency of a weakly compact cardinal, it is consistent
that there is a strongly inaccessible, non-weakly-compact cardinal $\kappa$ such that
the $\kappa$-Knaster property is
$\theta$-productive for all $\theta < \kappa$. On the other hand, the first author
and L\"{u}cke show in \cite{lh_lucke} that, if $\kappa$ is a regular uncountable
cardinal and the $\kappa$-Knaster property is infinitely productive, then
$\kappa$ is weakly compact in $\mathrm{L}$. The question as to whether the
$\kappa$-Knaster property can consistently be infinitely productive when
$\kappa$ is a successor cardinal (in particular, when $\kappa = \aleph_2$) is
raised but left unanswered in \cite{lh_lucke}. It is resolved negatively by
the following lemma, together with the fact (see Corollary~\ref{omegacolors})
that $\U(\lambda^+, \lambda^+, \omega, \omega)$ holds for every infinite cardinal $\lambda$.

\begin{lemma}\label{knasterlemma}
  Suppose that $\chi, \theta\in\reg(\kappa)$ and that $\kappa$ is $({<}\chi)$-inaccessible.
  For every coloring $c:[\kappa]^2\rightarrow\theta$ witnessing $\U(\kappa,\mu,\theta,\chi)$,
  there exists a corresponding poset $\mathbb P$ such that
  \begin{enumerate}
    \item $\mathbb P$ is well-met and $\chi$-directed closed with greatest lower bounds;
    \item if $\mu=2$, then $\mathbb P^\tau$ is $\kappa$-cc for all $\tau<\min(\{\chi, \theta\})$;
    \item if $\mu=\kappa$, then $\mathbb P^\tau$ has precaliber $\kappa$ for all $\tau<\min(\{\chi, \theta\})$;
    \item $\mathbb P^\theta$ is not $\kappa$-cc.
  \end{enumerate}
\end{lemma}

\begin{proof} This is a straightforward variation of the proof that $\Pr_1(\ldots)$ entails counterexamples to productivity of the chain condition.
  Let $c:[\kappa]^2 \rightarrow \theta$ witness $\U(\kappa, \mu, \theta, \chi)$, and
  let
  \[
  \mathbb P:=\{ (i,x)\mid i<\theta, ~ x\in[\kappa]^{<\chi}, ~ (c``[x]^2)\cap i=\emptyset\}\cup\{\one\},
  \]
  with $(i,x) \le_{\mathbb P} (j,y)$ iff $i=j$ and $x\supseteq y$, and with $(i,x)\le_{\mathbb P}\one$ for all $(i,x)$.
  For ease of notation, for each $p = (i,x) \in \mathbb{P}$, we let $i_p$ denote $i$ and $x_p$ denote $x$.
  Clearly, $\mathbb P$ is well-met and $\chi$-directed closed with greatest lower bounds.

\begin{claim}  $\mathbb P^\theta$ has an antichain of size $\kappa$.
\end{claim}
\begin{proof}
  We shall prove a slightly stronger result.
  Define the support, $\supp(q)$, of a condition $q\in\mathbb P^\theta$, by letting
  $\supp(q):=\{ j<\theta\mid q(j)\neq\one\}$.
  Let $J\in[\theta]^\theta$ be arbitrary.
  We now prove that $\{ q\in\mathbb P^\theta\mid \supp(q)=J\}$ has an antichain of size $\kappa$.
  For each $\alpha<\kappa$, define an element $q_\alpha$ in $\mathbb P^\theta$ by
  letting for all $j < \theta$:
  $$q_\alpha(j):=\begin{cases}(j,\{\alpha\})&\text{if }j\in J;\\\one&\text{if }j\notin J.\end{cases}$$
  Then, for any pair $(\alpha,\beta)\in[\kappa]^2$, we get that $q_\alpha(j)$ and
  $q_\beta(j)$ are incompatible in $\mathbb{P}$ for $j:=\min(J\setminus c(\alpha,\beta))$.
  Consequently, $\{q_\alpha\mid \alpha<\kappa\}$ is an antichain in $\{ q\in\mathbb P^\theta\mid \supp(q)=J\}$.
\end{proof}

Next, to prove Clauses (2) and (3), let $\tau<\min(\{\chi, \theta\})$,
and let  $A$ be an arbitrary $\kappa$-sized subset of $\mathbb P^\tau$.
Without loss of generality, $A\s (\mathbb P\setminus\{\one\})^\tau$.

    For each $q:\tau\rightarrow\mathbb P$ in $A$,
    let $i^q:=\langle i_{q(j)}\mid j<\tau\rangle$ and $x^q:=\bigcup\{ x_{q(j)} \mid j < \tau \}$.
    Since $\tau<\min(\{\chi, \theta\})$, and by the regularity of $\chi$ and $\theta$,
    we have $\sup(\im(i^q)) < \theta$ and $x^q\in[\kappa]^{<\chi}$ for all $q \in A$.
    Using the fact that $\kappa$ is $({<}\chi)$-inaccessible, fix $A'\in[A]^\kappa$ such that
    \begin{itemize}
      \item $\{ i^q\mid q\in A'\}$ is a singleton, say, $\{i^*\}$;
      \item $\{ x^q \mid q \in A' \}$ forms a head-tail-tail $\Delta$-system, with root, say, $r$;
      \item $q\mapsto \langle x_{q(j)}\cap r\mid j<\tau\rangle$ is constant over $A'$.
    \end{itemize}

    Let $\mathcal A:=\{ x^q \setminus r \mid q \in A' \}$. Since $c$ witnesses
    $\U(\kappa, \mu, \theta, \chi)$, we may find $\mathcal B\in[\mathcal A]^\mu$
    such that, for all $(b,b')\in[\mathcal B]^2$,
    we have $\min(c[b\times b'])>\sup(i^*)$. Fix $B\in[A']^\mu$ such that $\{ x^q\setminus r\mid q\in B\}=\mathcal B$.
    We claim that $B$ is centered.

    Fix a finite subset $\{q_m \mid m < n\}$ of $B$. Define a function
    $q$ with domain $\tau$ by letting $q(j) := (i^*(j),
    \bigcup_{m < n} x_{q_m(j)})$ for all $j < \tau$. To prove
    that $q$ is a lower bound for $\{q_m \mid m < n\}$, it suffices to
    verify that $q \in \mathbb{P}^\tau$. If not, then there are $m, m' < n$,
    $j < \tau$, and $\alpha < \alpha' < \kappa$ such that
    \begin{itemize}
      \item $\alpha \in x_{q_m(j)} \setminus x_{q_{m'}(j)}$;
      \item $\alpha' \in x_{q_{m'}(j)} \setminus x_{q_m(j)}$;
      \item $c(\alpha, \alpha') < i^*(j)$.
    \end{itemize}
    Since $x_{q_m(j)} \cap r = x_{q_{m'}(j)} \cap r$, it follows
    that $\alpha \in x_{q_m(j)} \setminus r$ and $\alpha' \in
    x_{q_{m'}(j)} \setminus r$. But $x_{q_m(j)} \setminus r,
    x_{q_{m'}(j)} \setminus r \in \mathcal{B}$, so we have
    $c(\alpha, \alpha') > \sup(\im(i^*)) \geq i^*(j)$.
    Thus, $B$ is centered, as desired.
\end{proof}

\begin{cor}
  If $\kappa$ is a regular uncountable cardinal and $\U(\kappa,\kappa,\omega,\omega)$
  holds, then there exists a $\kappa$-Knaster poset $\mathbb P$ such that $\mathbb P^\omega$
  is not $\kappa$-cc.\qed
\end{cor}

\subsection{Forcing axioms} \label{forcing_axioms_subsection}

Beginning in the 1970s, much work has been done attempting to generalize Martin's
Axiom to higher cardinals, and to $\aleph_2$ in particular. Versions of such a
generalization were obtained in unpublished work of both Laver and Baumgartner,
and a stronger version was obtained by Shelah in \cite{shelah_gma}. We state here
the version due to Baumgartner. We denote the axiom by $\BA$; more information
regarding $\BA$ can be found in \cite{tall_gma}.

\begin{defn}
  Let $\mathbb{P}$ be a forcing poset, and let $\nu$ be a cardinal. $\FA_\nu(\mathbb{P})$
  is the assertion that, whenever $\{D_\alpha \mid \alpha < \nu\}$ is a collection
  of dense subsets of $\mathbb{P}$, then there is a filter $G \subseteq \mathbb{P}$
  such that, for all $\alpha < \nu$, $G \cap D_\alpha \neq \emptyset$.
\end{defn}

\begin{defn}[Baumgartner's Axiom]
  $\BA$ is the statement that, if $\mathbb{P}$ is a poset that is well-met,
  countably closed, and $\aleph_1$-linked, then $\FA_\nu(\mathbb{P})$ holds
  for all $\nu < 2^{\aleph_1}$.
\end{defn}

\begin{fact}[Baumgartner] \label{ba_thm}
  Suppose that $\ch$ holds and $\kappa \geq \aleph_2$ is regular. Then there is a
  cofinality-preserving forcing extension in which $\BA + \ch + 2^{\aleph_1} = \kappa$
  holds.
\end{fact}

Shelah and Stanley, in \cite{shelah_stanley_gma}, prove that Fact~\ref{ba_thm}
fails if $\BA$ is weakened by omitting the requirement that $\mathbb{P}$ be well-met.
In particular, they prove the following result. (They prove the result for
$\lambda = \aleph_1$, but their proof generalizes.)

\begin{fact}[Shelah-Stanley, \cite{shelah_stanley_gma}]
  Suppose that $\lambda$ is an uncountable cardinal and $\lambda^{<\lambda} = \lambda$.
  Then there is a poset $\mathbb{P}$ of size $\lambda^+$ that is $\lambda$-closed
  and $\lambda$-linked but for which $\FA_{\lambda^+}(\mathbb{P})$ fails.
\end{fact}

The work in this paper was partially motivated by the following unpublished result
of Inamdar, which indicates another way in which $\BA$ cannot consistently be changed.
In particular, the requirement that $\mathbb{P}$ be $\aleph_1$-linked cannot
be replaced by the requirement that $\mathbb{P}$ has precaliber $\aleph_2$.
We would like to thank Inamdar for allowing us to include this theorem.

\begin{thm}[Inamdar, \cite{tanmay_note}]
  Suppose that $\lambda=\lambda^{<\lambda}$ is a regular uncountable cardinal.
  Then there is a forcing poset $\mathbb{Q}$ of size $\lambda^+$ such that
  \begin{enumerate}
    \item $\mathbb{Q}$ is well-met and $\lambda$-directed closed with greatest lower bounds;
    \item $\mathbb{Q}$ has precaliber $\lambda^+$;
    \item $\FA_{\lambda^+}(\mathbb{Q})$ fails.
  \end{enumerate}
\end{thm}

\begin{proof}
  By Corollary~\ref{omegacolors} (see also Corollary~\ref{l23}), we can fix a function $c:[\lambda^+]^2
  \rightarrow \omega$ witnessing $\U(\lambda^+, \lambda^+, \omega, \lambda)$.
  For each $i < \omega$, define a poset $\mathbb{P}_i := \{x \in [\lambda^+]^{<\lambda}
  \mid c``[x]^2 \cap i = \emptyset\}$, ordered by reverse inclusion, and let
  $\mathbb{Q}_i$ be the ${<}\lambda$-support product of $\lambda$ copies of
  $\mathbb{P}_i$. It is immediate that, for all $i < \omega$, $\mathbb{Q}_i$
  has size $\lambda^+$ and is well-met and $\lambda$-directed closed with
  greatest lower bounds.

  \begin{claim}
    For all $i < \omega$, $\mathbb{Q}_i$ has precaliber $\lambda^+$.
  \end{claim}

  \begin{proof}
    The proof is essentially the same as that of Clause~(3) of Lemma~\ref{knasterlemma}.
  \end{proof}

  \begin{claim}
    Suppose that $i < \omega$ and $\FA_{\lambda^+}(\mathbb{Q}_i)$ holds. Then
    $\mathbb{P}_i$ is $\lambda$-centered.
  \end{claim}

  \begin{proof}
    For each $p \in \mathbb{P}_i$, the set $D_p := \{q \in \mathbb{Q}_i \mid \text{for some }
    \eta \in \dom(q), ~ q(\eta) \leq_{\mathbb{P}_i} p\}$ is dense in
    $\mathbb{Q}_i$. As $|\mathbb{P}_i| = \lambda^+$, there is a filter
    $G \subseteq \mathbb{Q}_i$ such that $G \cap D_p \neq \emptyset$ for all
    $p \in \mathbb{P}_i$. For $\eta < \lambda$, let $G_\eta$ be the upward
    closure of $\{q(\eta) \mid q \in G\}$ in $\mathbb{P}_i$. Then each $G_\eta$
    is a centered subset of $\mathbb{P}_i$, and $\bigcup_{\eta < \lambda} G_\eta
    = \mathbb{P}_i$, so $\mathbb{P}_i$ is $\lambda$-centered.
  \end{proof}

  \begin{claim}
    Suppose that $\mathbb{P}_i$ is $\lambda$-centered for all $i < \omega$.
    Then $\prod_{i < \omega} \mathbb{P}_i$ is $\lambda$-centered.
  \end{claim}

  \begin{proof}
    For each $i < \omega$, let $\{G^i_\eta \mid \eta < \lambda\}$ be a collection
    of centered subsets that covers $\mathbb{P}_i$. For each $h \in {^\omega}\lambda$, let
    \[
      G_h := \left\{p \in \prod_{i < \omega} \mathbb{P}_i \Mid \text{for all }
      i < \omega, ~ p(i) \in G^i_{h(i)}\right\}.
    \]
    Then $\{G_h \mid h \in {^\omega}\lambda\}$ is a collection of centered subsets
    that covers $\mathbb{P}_i$. Since $\lambda^{\aleph_0} = \lambda$, the collection
    has size $\lambda$ and hence witnesses that $\prod_{i < \omega} \mathbb{P}_i$
    is $\lambda$-centered.
  \end{proof}

  Therefore, if $\FA_{\lambda^+}(\mathbb{Q}_i)$ holds for all $i < \omega$,
  then $\prod_{i < \omega} \mathbb{P}_i$ is $\lambda$-centered and hence
  has the $\lambda^+$-cc. However, if, for all $\alpha < \lambda^+$, we let
  $p_\alpha \in \prod_{i < \omega} \mathbb{P}_i$ be the constant function taking
  value $\{\alpha\}$, then, as in the proof of Lemma~\ref{knasterlemma}, $\{p_\alpha \mid \alpha
  < \lambda^+\}$ is an antichain of size $\lambda^+$ in $\prod_{i < \omega}
  \mathbb{P}_i$. It follows that there is $i < \omega$ for which $\FA_{\lambda^+}(\mathbb{Q}_i)$
  fails.
\end{proof}

\section{Closed colorings} \label{closed_section}

\subsection{Preliminaries} \label{closed_subsection}

In this section, we undertake a thorough analysis of witnesses to
$\U(\kappa, \mu, \theta, \chi)$ that satisfy certain \emph{closure}
conditions isolated by the following definition, and the circumstances
under which such colorings must exist. Our reasons for focusing on
closed colorings are twofold. Firstly, closed colorings behave more
nicely than general colorings. For example, as Lemma~\ref{pumpclosed}
will make clear, a closed witness to $\U(\kappa, 2, \theta, \chi)$
is actually a witness to $\U(\kappa, \kappa, \theta, \chi)$. Secondly,
closed colorings seem to arise naturally. Our primary methods for constructing
witnesses to $\U(\kappa, \mu, \theta, \chi)$ come from the techniques of
\emph{walks on ordinals}, and, as we shall see in this section, the colorings
that arise from these constructions tend to be closed.

\begin{defn}
  For a subset $\Sigma\s\kappa$, we say that $c:[\kappa]^2\rightarrow\theta$ is
  \emph{$\Sigma$-closed} if, for all $\beta<\kappa$ and $i\le \theta$,
  the set $D^c_{\le i}(\beta):=\{\alpha<\beta\mid c(\alpha,\beta)\le i\}$ satisfies
  $\acc^+(D^c_{\le i}(\beta))\cap\Sigma\s D^c_{\le i}(\beta)$.
  We say that $c$ is \emph{somewhere-closed} iff it is $\Sigma$-closed for some stationary $\Sigma\s\kappa$,
  that $c$ is \emph{tail-closed} iff it is $E^\kappa_{\ge\sigma}$-closed for some $\sigma\in\reg(\kappa)$,
  and that $c$ is \emph{closed} iff it is $\kappa$-closed.
\end{defn}

\begin{lemma}\label{pumpclosed}
  Suppose that $c:[\kappa]^2\rightarrow\theta$ is a coloring and $\omega\le\chi<\kappa$.
  Then $(1)\implies(2)\implies(3)$:
  \begin{enumerate}
    \item For some stationary $\Sigma\s E^\kappa_{\ge\chi}$, $c$ is a $\Sigma$-closed
      witness to $\U(\kappa,2,\theta,\chi)$.
    \item For every family $\mathcal A\s[\kappa]^{<\chi}$ consisting of $\kappa$-many
      pairwise disjoint sets, for every club $D\s\kappa$, and for every $i<\theta$,
      there exist $\gamma\in D$, $a\in\mathcal A$, and $\epsilon < \gamma$ such that:
      \begin{itemize}
        \item $\gamma < a$;
        \item for all $\alpha \in (\epsilon, \gamma)$ and all $\beta\in a$,
        we have $c(\alpha,\beta)>i$.
      \end{itemize}
      \item $c$ witnesses $\U(\kappa,\kappa,\theta,\chi)$.
  \end{enumerate}
\end{lemma}

\begin{proof}
  $(1)\implies(2)$
  Fix a family $\mathcal A\s[\kappa]^{<\chi}$ consisting of $\kappa$-many pairwise disjoint sets,
  a club $D \subseteq \kappa$, and a color $i<\theta$.
  Find $\mathcal X\s[\kappa]^{<\chi}$ consisting of $\kappa$-many pairwise disjoint sets
  such that every $x\in\mathcal X$ is of the form $\{\gamma\} \cup a$ for some
  $\gamma \in \Sigma \cap D$ and $a \in \mathcal{A}$. As $c$ witnesses $\U(\kappa,2,\theta,\chi)$,
  we may pick $(x,y)\in[\mathcal X]^2$ such that $\min(c[x\times y])>i$.
  Fix $\gamma\in(\Sigma\cap D)\cap x$ and $a\in\mathcal A\cap\mathcal P(y)$.
  Clearly, $\gamma < a$ and $|a|<\chi\le\cf(\gamma)$.
  Now, let $\beta\in a$ be arbitrary. Since $(\gamma,\beta)\in x\times y$, we have $c(\gamma,\beta)>i$,
  and, since $\gamma\in\Sigma$, there must exist $\epsilon(\gamma,\beta)<\gamma$ such that,
  for all $\alpha\in(\epsilon(\gamma,\beta),\gamma)$, $c(\alpha,\beta)>i$. Since
  $\cf(\gamma) \geq \chi > |a|$, we have that $\epsilon := \sup\{\epsilon(\gamma, \beta)
  \mid \beta \in a\}$ is less than $\gamma$. Then $\gamma$, $a$, and $\epsilon$
  are as sought.

  $(2)\implies(3)$
  Fix a family $\mathcal A\s[\kappa]^{<\chi}$ consisting of $\kappa$-many pairwise disjoint
  sets and a color $i<\theta$. Let $\Gamma$ denote the collection of all $\gamma<\kappa$
  such that for some $a_\gamma\in\mathcal A$ and some $\epsilon_\gamma < \gamma$, we have
  \begin{itemize}
    \item $\gamma < a_\gamma$; and
    \item for all $\alpha \in (\epsilon_\gamma, \gamma)$ and all $\beta\in a_\gamma$,
      we have $c(\alpha,\beta)>i$.
  \end{itemize}
  By the hypothesis, $\Gamma$ is stationary.
  Define $f:\Gamma\rightarrow\kappa$ and $g:\Gamma\rightarrow\kappa$ by letting, for all $\gamma\in\Gamma$,
  $f(\gamma):=\epsilon_\gamma$ and $g(\gamma):=\sup(a_\gamma)$.
  By Fodor's Lemma, we now pick $\epsilon<\kappa$ for which
  $B:=\{ \gamma\in \Gamma\mid f(\gamma)=\epsilon\ \&\ g[\gamma]\s\gamma\}$ is stationary,
  and then let $\mathcal B:=\{ a_\gamma\mid \gamma\in B\}$.
  We claim that $\min(c[a \times b]) > i$ for all $(a,b) \in [\mathcal{B}]^2$.
  To this end, fix $(a,b)\in[\mathcal B]^2$, and let $\gamma, \delta\in B$
  be such that $a=a_\gamma$ and $b=a_\delta$. Then we have
  $$
    \epsilon_\delta = \epsilon<\gamma<\alpha<\delta<\beta,
  $$
  and hence $c(\alpha,\beta)>i$.
\end{proof}

\subsection{Walks on ordinals} \label{walks_subsection}

We now introduce some of the machinery we will need to construct witnesses to
$\U(\kappa, \mu, \theta, \chi)$ using the techniques of walks on ordinals.

\begin{defn}
  A \emph{$C$-sequence} over $\kappa$ is a sequence $\langle C_\alpha \mid \alpha
  < \kappa \rangle$ such that, for all $\alpha < \kappa$,
  \begin{itemize}
    \item $C_0 = \emptyset$;
    \item $C_{\alpha + 1} = \{\alpha\}$;
    \item if $\alpha \in \acc(\kappa)$, then $C_\alpha$ is a club in $\alpha$.
  \end{itemize}
\end{defn}

\begin{defn}[\cite{MR908147}]\label{walks}
  Given a $C$-sequence $\langle C_\alpha\mid\alpha<\kappa\rangle$, we derive various functions as follows.
  For all $\alpha<\beta<\kappa$,
  \begin{itemize}
    \item $\Tr(\alpha,\beta)\in{}^\omega\kappa$ is defined recursively by letting, for all $n<\omega$,
      $$\Tr(\alpha,\beta)(n):=\begin{cases}
      \beta & \text{if } n=0;\\
      \min(C_{\Tr(\alpha,\beta)(n-1)}\setminus\alpha) & \text{if } n>0\ \&\ \Tr(\alpha,\beta)(n-1)>\alpha;\\
      \alpha & \text{otherwise};
      \end{cases}
      $$
    \item (Number of steps) $\rho_2(\alpha,\beta):=\min\{n<\omega\mid \Tr(\alpha,\beta)(n)=\alpha\}$;
    \item (Upper trace) $\tr(\alpha,\beta):=\Tr(\alpha,\beta)\restriction \rho_2(\alpha,\beta)$.
  \end{itemize}
\end{defn}
\begin{remark}
  To avoid notational confusion, note that there is no relationship between the two-place instance $\Tr(\alpha,\beta)$
  and the one-place instance $\Tr(S)$.
\end{remark}

\begin{defn}[\cite{paper18}]
  For $\alpha<\beta<\kappa$, let
  $$\lambda_2(\alpha,\beta):=\sup(\alpha\cap\{ \sup(C_\delta\cap\alpha)
  \mid \delta \in \im(\tr(\alpha, \beta))\}).$$
\end{defn}

Note that $\lambda_2(\alpha, \beta) < \alpha$ whenever $0<\alpha < \beta < \kappa$.
To motivate the preceding definition, let us point out the following lemma.

\begin{lemma}\label{lambda2}
  Suppose that $\lambda_2(\gamma,\beta)<\alpha<\gamma<\beta<\kappa$.
  Then $\tr(\gamma,\beta)\sq \tr(\alpha,\beta)$ and one of the following cases holds:
  \begin{enumerate}
    \item $\gamma\in\im(\tr(\alpha,\beta))$; or
    \item $\gamma\in\acc(C_\delta)$ for $\delta:=\min(\im(\tr(\gamma,\beta)))$.
      In particular, $\gamma\in\acc(C_\delta)$ for some  $\delta\in\im(\tr(\alpha,\beta))$.
  \end{enumerate}
\end{lemma}

\begin{proof}
  We first show, by induction on $i$, that $\Tr(\alpha, \beta)(i) = \Tr(\gamma, \beta)(i)$
  for all $i < \rho_2(\gamma, \beta)$, i.e., that $\tr(\gamma,\beta)\sq \tr(\alpha,\beta)$.
  Clearly, $\Tr(\alpha,\beta)(0)=\beta=\Tr(\gamma,\beta)(0)$. Next, suppose that
  $i + 1 < \rho_2(\gamma, \beta)$ and $\Tr(\alpha, \beta)(i) = \Tr(\gamma, \beta)(i)$.
  Since $i + 1 < \rho_2(\gamma, \beta)$, it must be the case that $\gamma
  \notin C_{\Tr(\gamma, \beta)(i)}$, and therefore $\sup(C_{\Tr(\gamma, \beta)(i)}
  \cap \gamma) \leq \lambda_2(\gamma, \beta) < \alpha$. It follows that
  \[
    \min(C_{\Tr(\alpha,\beta)(i)}\setminus\alpha)=\min(C_{\Tr(\gamma,\beta)(i)}\setminus\alpha)=
    \min(C_{\Tr(\gamma,\beta)(i)}\setminus\gamma),
  \]
  and hence $\Tr(\alpha,\beta)(i+1)=\Tr(\gamma,\beta)(i+1)$.

  To prove the second part of the lemma, set $n:=\rho_2(\gamma,\beta)-1$
  and $\delta := \tr(\gamma, \beta)(n) = \tr(\alpha, \beta)(n)$,
  note that $\gamma \in C_\delta$, and consider the following two cases.

  $\br$ If $\gamma\in\nacc(C_{\delta})$, then $\sup(C_\delta \cap \gamma)
  \leq \lambda_2(\gamma, \beta) < \alpha$, so $\gamma=\tr(\alpha,\beta)(n+1)$
  and we are in Case (1) of the statement of the lemma.

  $\br$ If $\gamma\in\acc(C_\delta)$, then we are in Case (2) of the statement
  of the lemma.
\end{proof}

\begin{cor}\label{rho2_closed}
  $\rho_2:[\kappa]^2\rightarrow\omega$ is closed.
\end{cor}

\begin{proof}
  Fix $\beta<\kappa$, $i<\omega$, and $A\s D^{\rho_2}_{\le i}(\beta)$ with $\gamma := \sup(A)$ smaller than $\beta$.
  Fix $\alpha\in A$ above $\lambda_2(\gamma,\beta)$. Then, by Lemma~\ref{lambda2},
  $\rho_2(\gamma,\beta)\le \rho_2(\alpha,\beta) \le i$, so $\gamma\in D^{\rho_2}_{\le i}(\beta)$.
\end{proof}

The following corollary now follows from \cite[Theorem~6.3.6]{todorcevic_book}.

\begin{cor}[Todorcevic]\label{omegacolors} For every infinite cardinal $\lambda$, there is a closed witness to $\U(\lambda^+,\lambda^+,\omega,\cf(\lambda))$.\qed
\end{cor}

\begin{remark}The statement of \cite[Theorem~6.3.6]{todorcevic_book} in that source has a typing error,
where ``of size $<\kappa$'' should have been ``of size ${<}\cf(\kappa)$''.
For example, by Theorem~\ref{t310} above, if $\lambda$ is a singular limit
of strongly compact cardinals and $\lambda$ has uncountable cofinality, then
$\U(\lambda^+, \lambda^+, \omega, \cf(\lambda)^+)$ fails.
\end{remark}

\subsection{The first construction} \label{first_construction_subsection}

We are now ready to begin constructing closed witnesses to $\U(\kappa, \mu, \theta, \chi)$
using walks on ordinals. Our first result shows that the existence of closed witnesses to
$\U(\kappa, \kappa, \theta, \chi)$ follows from the existence of
certain strong counterexamples to $\refl(\theta, E^\kappa_{\geq \chi})$.

\begin{thm}\label{simu}
  Suppose that $\theta\in\reg(\kappa)$ and there exist a sequence
  $\langle H_i\mid i<\theta\rangle$ of pairwise disjoint subsets of $\kappa$
  and a $C$-sequence $\vec C=\langle C_\alpha\mid\alpha<\kappa\rangle$ such that,
  for every $\alpha\in\acc(\kappa)$,
  $$\sup\{ i<\theta\mid \acc(C_\alpha)\cap H_i\neq\emptyset\}<\theta.$$
  Then there exists a closed coloring $c:[\kappa]^2\rightarrow\theta$ such that
  \begin{enumerate}
    \item for every $\chi\in\reg(\kappa)$ for which
      $$
        \sup\{ i<\theta\mid E^\kappa_{\ge\chi}\cap H_i\text{ is stationary}\}=\theta,
      $$
      $c$ witnesses $\U(\kappa,\kappa,\theta,\chi)$;
    \item if $|\{ C_\alpha\cap\beta\mid \alpha<\kappa\}|<\kappa$ for all $\beta<\kappa$,
      then  $\mathcal T(c)$ is a $\kappa$-tree.
  \end{enumerate}
\end{thm}

\begin{proof}
  Let $\tr:[\kappa]^2\rightarrow{}^{<\omega}\kappa$ denote the upper trace function
  along $\vec C$ (recall Definition~\ref{walks}).
  Define a function $h:\kappa\rightarrow\theta$ by letting, for all $\gamma < \kappa$,
  $$
    h(\gamma):=\sup\{ i<\theta\mid (\acc(C_{\gamma})\cup\{\gamma\})\cap H_i\neq\emptyset\}.
  $$
  Then, define a coloring $c:[\kappa]^2\rightarrow\theta$ by letting, for all $\alpha<\beta<\kappa$,
  $$
    c(\alpha,\beta):=\max\{ h(\tau)\mid \tau\in\im(\tr(\alpha,\beta))\}.
  $$

  \begin{claim}
    $c$ is closed.
  \end{claim}

  \begin{proof}
    Suppose that $\beta < \kappa$, $i < \theta$, and $A \subseteq D^c_{\leq i}(\beta)$,
    with $\gamma := \sup(A)$ smaller than $\beta$. Fix $\alpha \in A$ above $\lambda_2(\gamma, \beta)$.
    By Lemma~\ref{lambda2}, $\im(\tr(\gamma,\beta))\s\im(\tr(\alpha,\beta))$,
    and hence, by the definition of $c$, we have $c(\gamma, \beta) \leq c(\alpha, \beta) \leq i$,
    so $\gamma \in D^c_{\leq i}(\beta)$.
  \end{proof}

  \begin{claim} Suppose that $\chi\in\reg(\kappa)$ and
    $$
      \sup\{ i<\theta\mid E^\kappa_{\ge\chi}\cap H_i\text{ is stationary}\}=\theta.
    $$
    Then $c$ witnesses $\U(\kappa,\kappa,\theta,\chi)$.
    Furthermore, for every $i<\theta$,  every $\chi'<\chi$
    and every sequence  $\langle a_\gamma\mid\gamma<\kappa\rangle$ with
    $a_\gamma\in[\kappa\setminus\gamma]^{\chi'}$ for each $\gamma<\kappa$,
    there exists a \emph{stationary} $S\s\kappa$ such that
    $\min(c[a_\gamma\times a_{\gamma'}])>i$ for all $(\gamma,\gamma')\in[S]^2$.
  \end{claim}

  \begin{proof}
    Let $i<\theta$ and $\langle a_\gamma\mid\gamma<\kappa\rangle$ be as above.
    Find $j>i$ such that $H_j\cap E^\kappa_{\ge\chi}$ is stationary, and let
    $\Gamma:=H_j\cap E^\kappa_{\ge\chi}$.
    Define $f:\Gamma\rightarrow\kappa$ and $g:\Gamma\rightarrow\kappa$ by letting, for all $\gamma\in\Gamma$,
    \begin{itemize}
    \item $f(\gamma):=\sup\{ \lambda_2(\gamma,\beta)\mid \beta\in a_\gamma\setminus\{\gamma\}\}$;
    \item $g(\gamma):=\sup(a_\gamma)$.
    \end{itemize}
    For each $\gamma\in\Gamma$, $\cf(\gamma) \geq \chi > |a_\gamma|$, so $f$ is regressive.
    By Fodor's Lemma, we now pick $\epsilon<\kappa$ for which
    $S:=\{ \gamma\in \Gamma\mid f(\gamma)=\epsilon\ \&\ g[\gamma]\s\gamma\}$ is stationary.
    To see that $S$ is as sought, fix an arbitrary pair $(\gamma,{\gamma'})\in[S]^2$
    and an arbitrary pair $(\alpha,\beta)\in a_\gamma\times a_{\gamma'}$.
    There are two cases to consider.

    $\br$ If $\gamma'\in\im(\tr(\alpha,\beta))$, then since $\gamma'\in H_j$, we have $c(\alpha,\beta)\ge h({\gamma'})\ge j>i$.

    $\br$ Otherwise, we have
    $$
      \lambda_2({\gamma'},\beta)\le f({\gamma'}) = \epsilon<\gamma\le\alpha<{\gamma'}<\beta,
    $$
    so by Lemma~\ref{lambda2}, it must be the case that
    there exists $\delta\in\im(\tr(\alpha,\beta))$ such that
    $\gamma'\in\acc(C_\delta)$. Altogether, $\gamma'\in \acc(C_\delta)\cap H_j$, and hence
    $c(\alpha,\beta)\ge h(\delta) \ge j >i$.
  \end{proof}

  The following claim will now complete the proof of the theorem.

  \begin{claim} \label{claim493}
    Suppose that $|\{ C_\alpha\cap\beta\mid \alpha<\kappa\}|<\kappa$ for all $\beta<\kappa$.
    Then $\mathcal T(c)$ is a $\kappa$-tree.
  \end{claim}

  \begin{proof} \renewcommand{\qedsymbol}{\ensuremath{\square\ \square}}
    Suppose not. Pick $\beta<\kappa$ for which $\{ c(\cdot,\gamma)\restriction\beta\mid
    \gamma<\kappa\}$ has size $\kappa$. Pick $\Gamma\in[\kappa]^\kappa$ on which the map
    $\gamma\mapsto c(\cdot,\gamma)\restriction\beta$ is one-to-one. Pick a pair
    $(\gamma_0,\gamma_1)\in[\Gamma]^2$ for which $\rho_2(\beta, \gamma_0) = \rho_2(\beta, \gamma_1)$
    and the following equalities hold:
    \begin{itemize}
      \item $\langle h(\tr(\beta,\gamma_0)(n))\mid n<\rho_2(\beta,\gamma_0)\rangle=\langle
        h(\tr(\beta,\gamma_1)(n))\mid n<\rho_2(\beta,\gamma_1)\rangle$;
      \item $\langle C_{\tr(\beta,\gamma_0)(n)}\cap\beta\mid n\le\rho_2(\beta,\gamma_0)\rangle=
        \langle C_{\tr(\beta,\gamma_1)(n)}\cap\beta\mid n\le\rho_2(\beta,\gamma_1)\rangle$.
    \end{itemize}

    For notational simplicity, write $\langle i^n\mid n<m\rangle$ and
    $\langle C^n\mid n\le m\rangle$ for the above common sequences.
    We shall meet a contradiction by showing that for each $\alpha<\beta$,
    $c(\alpha,\gamma_0)=c(\alpha,\gamma_1)$. Let $\alpha<\beta$ be arbitrary.
    The analysis splits into two cases.

    $\br$ If $C^n\cap[\alpha,\beta)=\emptyset$ for all $n\le m$, then
    $\min(C^n\setminus\alpha)=\min(C^n\setminus\beta)$ for all $n\le m$.
    But then, for each $j<2$, $\Tr(\alpha,\gamma_j)(n)=\Tr(\beta,\gamma_j)(n)$
    for all $n\le m$. It follows that $\Tr(\alpha,\gamma_j)(m)=\beta$,
    $\tr(\alpha,\gamma_j)=\tr(\beta,\gamma_j){}^\smallfrown\tr(\alpha,\beta)$, and
    $$
      c(\alpha,\gamma_j)=\max\{c(\beta,\gamma_j),c(\alpha,\beta)\}=\max\{i^n,c(\alpha,\beta)\mid n<m\}.
    $$
    As the preceding expression does not depend on $j$, we infer that $c(\alpha,\gamma_0)=c(\alpha,\gamma_1)$.

    $\br$ Otherwise, let $n\le m$ be least
    such that $C^n\cap[\alpha,\beta)\neq\emptyset$.
    Then, as in the previous analysis, for each $j<2$, we have $\Tr(\alpha,\gamma_j)\restriction n+1=
    \Tr(\beta,\gamma_j)\restriction n+1$ and $\Tr(\alpha,\gamma_j)(n+1)=\min(C^n\setminus\alpha)$.
    Let $\eta:=\min(C^n\setminus\alpha)$. Then, for each $j < 2$, we have
    $\tr(\alpha, \gamma_j) = \tr(\eta, \gamma_j){}^\smallfrown \tr(\alpha, \eta)$, so
    \[
      c(\alpha,\gamma_0)= \max\{i^k,c(\alpha,\eta)\mid k\le n\}=c(\alpha,\gamma_1).\qedhere
    \]
  \end{proof}
  \let\qed\relax
\end{proof}

We can immediately derive a number of corollaries from Theorem~\ref{simu}.
The first provides, among other things, a complete answer to the question of
the existence of closed witnesses to $\U(\ldots)$ at successors of regular cardinals.

\begin{cor}\label{l23}
  Suppose that $\theta, \chi\in\reg(\kappa)$. Then $(1)\implies(2)\implies(3)\implies(4)$:
  \begin{enumerate}
    \item $\kappa$ is the successor of a regular cardinal.
    \item $E^\kappa_{\ge\chi}$ admits a non-reflecting stationary set.
    \item There exists a stationary $S\s E^\kappa_{\ge\chi}$ such that
      $\Tr(S)\cap E^\kappa_{\ge\theta}$ is nonstationary.
    \item There exists a closed witness to $\U(\kappa,\kappa,\theta,\chi)$.
  \end{enumerate}
\end{cor}

\begin{proof}
  $(1)\implies(2)\implies(3)$ Trivial.

  $(3)\implies(4)$ Let $S$ be as in (3). Fix a club $D$ in $\kappa$
  disjoint from $\Tr(S)\cap E^\kappa_{\ge\theta}$ and a $C$-sequence
  $\vec C=\langle C_\alpha\mid\alpha<\kappa\rangle$ such that for all $\alpha\in\acc(\kappa)$,
  \begin{itemize}
    \item $\otp(C_\alpha)=\cf(\alpha)$, and
    \item if $S\cap D\cap \alpha$ is nonstationary in $\alpha$, then $S\cap D\cap\acc(C_\alpha)=\emptyset$.
  \end{itemize}
  Let $\langle H_i\mid i<\theta\rangle$ be an arbitrary partition of $S\cap D$ into stationary sets.
  Now appeal to Theorem~\ref{simu}.
\end{proof}

\begin{remark}
  If $\kappa$ is the successor of a regular cardinal $\lambda$,  $\square^*_\lambda$ holds,
  and $\theta \in \reg(\lambda^+)$, then we in fact obtain the existence of a closed witness $c$ to
  $\U(\kappa, \kappa, \theta, \lambda)$ for which $\mathcal{T}(c)$ is a $\kappa$-Aronszajn tree.
  Likewise, if $\kappa$ is the successor of a singular cardinal $\lambda=2^{<\lambda}$,
  $\square^*_\lambda$ holds, and $E^\kappa_{\ge\chi}$ admits a non-reflecting stationary set,
  then, for each $\theta \in \reg(\kappa)$, we obtain the existence of a closed witness $c$ to
  $\U(\kappa, \kappa, \theta, \chi)$ for which $\mathcal{T}(c)$ is a $\kappa$-Aronszajn tree.
  The proof of these facts comes from a straightforward combination of the proofs of
  Theorem~\ref{simu}, Lemma~\ref{lemma610} and Corollary~\ref{l23}, so we omit it.
\end{remark}

It follows that, in the Laver-Shelah model from \cite{laver_shelah} in which $\ch$
(and hence $\square^*_{\aleph_1}$) holds and in which every $\aleph_2$-Aronszajn tree
is special, for each $n<2$, there is a closed witness $c$ to
$\U(\aleph_2,\aleph_2,\aleph_n,\aleph_1)$ for which $\mathcal T(c)$ is a special
$\aleph_2$-Aronszajn tree. Let us show that the other extreme is consistent as well.
Recall that a $\lambda^+$-tree $\mathcal{T}$ is \emph{almost Souslin} if, for
every antichain $A$ of $\mathcal{T}$, the set of $\alpha \in E^{\lambda^+}_{\cf(\lambda)}$ such that $A$ has
non-empty intersection with the $\alpha^{\mathrm{th}}$ level of $\mathcal{T}$ is
nonstationary.

\begin{cor}\label{vel}
  If $\mathrm{V}=\mathrm{L}$, then for every infinite regular cardinal $\lambda$ and every
  $\theta\in\reg(\lambda^+)$, there exists a closed witness $c$ to
  $\U(\lambda^+,\lambda^+,\theta,\lambda)$ for which $\mathcal T(c)$ is an
  almost Souslin $\lambda^+$-Aronszajn tree.
\end{cor}

\begin{proof}
  Suppose that $\mathrm{V}=\mathrm{L}$ and that $\lambda$ is a infinite regular cardinal.
  Suppose also that $\lambda$ is uncountable (if $\lambda = \aleph_0$,
  then the proof is similar and slightly easier).
  By \cite[Theorem~3.1 and Lemma~2.6]{paper21}, $\sd_\lambda^*$ holds.
  Then, by \cite[Theorem~4.11]{paper21}, there exists a $C$-sequence
  $\vec C=\langle C_\alpha\mid\alpha<\lambda^+\rangle$ satisfying the following
  conditions:
  \begin{itemize}
    \item $\otp(C_\alpha)\le\lambda$ for all $\alpha<\lambda^+$;
    \item for every $\alpha<\lambda^+$ and $\bar\alpha\in\acc(C_\alpha)$,
      $C_{\bar\alpha}=C_\alpha\cap\bar\alpha$;
    \item for every stationary $S\s E^{\lambda^+}_\lambda$, there are
      $\beta'\in S$ and $\beta\in\nacc(C_{\beta'})\cap S$ such that
      $C_{\beta'}\cap\beta\sq C_\beta$ (i.e.,
      $C_\beta$ is an end-extension of $C_{\beta'}\cap\beta$).
  \end{itemize}

  Let $\kappa:=\lambda^+$, and fix $\theta\in\reg(\kappa)$. Let
  $\langle H_i\mid i<\theta\rangle$ be some partition of $E^{\kappa}_\lambda$
  into $\theta$-many stationary sets, and define functions $h$ and $c$ as in the
  proof of Theorem~\ref{simu}. Then $c$ witness $\U(\kappa,\kappa,\theta,\lambda)$,
  and $\mathcal T(c)$ is a $\kappa$-tree. By Proposition~\ref{aronszajn},
  $\mathcal T(c)$ is in fact a $\kappa$-Aronszajn tree.

  Finally, to see that $\mathcal T(c)$ is almost Souslin, suppose that
  $\langle t_\beta\mid \beta\in B\rangle$ is a sequence such that $B$ is a
  stationary subset of $E^{\lambda^+}_\lambda$ and $t_\beta\in \mathcal
  T(c)\cap{}^\beta\theta$ for each $\beta\in B$. We need to find
  $(\beta,\beta')\in[B]^2$ such that $t_\beta\s t_{\beta'}$.
  For each $\beta\in B$, pick $\gamma_\beta\in[\beta,\kappa)$ such that
  $t_\beta=c(\cdot,\gamma_\beta)\restriction\beta$. There are two cases
  (and a few subcases) to consider.

  $\br$ If $\gamma_\beta=\beta$ for stationarily many $\beta\in B$, then let us
  fix some $i<\theta$ for which the following set is stationary:
  $$
    S:=\{ \beta\in B\mid \beta=\gamma_\beta\ \&\ h(\beta)=i\}.
  $$
  By the choice of $\vec C$, we may pick $\beta'\in S$ and
  $\beta\in\nacc(C_{\beta'})\cap S$ such that $C_{\beta'}\cap\beta\sq C_\beta$.
  We claim that $t_\beta\s t_{\beta'}$. To show this, we now fix an arbitrary
  $\alpha<\beta$ and prove that $c(\alpha,\beta)=c(\alpha,\beta')$.
  Let $\tau:=\sup(C_{\beta'}\cap\beta)$, and note that $C_{\beta'}\cap(\tau+1)=C_\beta\cap(\tau+1)$.

  $\br\br$ If $\alpha\le\tau$, then $\tr(\alpha,\beta) = \tr(\alpha,\beta')$,
  so $c(\alpha,\beta)=c(\alpha,\beta')$.

  $\br\br$ If $\alpha > \tau$, then $\tr(\alpha,\beta')=\langle\beta'
  \rangle{}^\smallfrown\tr(\alpha,\beta)$, and thus $c(\alpha,\beta')=
  \max\{h(\beta'),c(\alpha,\beta)\}$. But $h(\beta')=h(\beta)\le c(\alpha,\beta)$,
  so $c(\alpha,\beta')=c(\alpha,\beta)$.

  $\br$ If $\gamma_\beta>\beta$ for club-many $\beta\in B$, then we may fix $i<\theta$,
  $\epsilon<\kappa$ and $t\in\mathcal T(c)\cap{}^{(\epsilon+1)}\theta$ for which the following set is stationary:
  $$S:=\left\{ \beta\in B\Mid \begin{gathered}\beta<\gamma_\beta\ \&\ c(\beta,\gamma_\beta)=i\ \&  \\\lambda_2(\beta,\gamma_\beta)=\epsilon\ \&\ t_\beta\restriction(\epsilon+1)=t\end{gathered}\right\}.$$
  By the choice of $\vec C$, we may pick $\beta'\in S$ and $\beta\in\nacc(C_{\beta'})\cap S$
  such that $C_{\beta'}\cap\beta\sq C_\beta$. We claim that $t_\beta\s t_{\beta'}$.
  To show this, we now fix an arbitrary $\alpha<\beta$ and prove that
  $c(\alpha,\gamma_\beta)=c(\alpha,\gamma_{\beta'})$.

  $\br\br$ If $\alpha\le\epsilon$, then $t_\beta(\alpha)=t(\alpha)=t_{\beta'}(\alpha)$.

  $\br\br$ If $\alpha > \epsilon$, then, since $\cf(\beta)=\cf(\beta')=\lambda$,
  neither $\beta$ nor $\beta'$ appears as an accumulation point of $C_\delta$ for any $\delta < \kappa$,
  so, by Lemma~\ref{lambda2},
  we have $\tr(\alpha,\gamma_\beta)=\tr(\beta,\gamma_\beta){}^\smallfrown\tr(\alpha,\beta)$
  and $\tr(\alpha,\gamma_{\beta'})=\tr(\beta',\gamma_{\beta'}){}^\smallfrown\tr(\alpha,\beta')$.
  Consequently, $c(\alpha,\gamma_\beta)=\max\{c(\beta,\gamma_\beta),c(\alpha,\beta)\}$ and
  $c(\alpha,\gamma_{\beta'})=\max\{c(\beta',\gamma_{\beta'}),c(\alpha,\beta')\}$.
  As $\beta\in\nacc(C_{\beta'})$ and $C_\beta'\cap\beta\sq C_\beta$, the previous analysis
  shows that $c(\alpha,\beta)=c(\alpha,\beta')$. By our choice of $S$, we also
  have $c(\beta, \gamma_\beta) = i = c(\beta', \gamma_{\beta'})$. Putting this together,
  we obtain $c(\alpha,\gamma_\beta)=c(\alpha,\gamma_{\beta'})$.
\end{proof}

\begin{cor}\label{l24}
  Suppose that $\sigma, \theta \in\reg(\kappa)$. If
  $\square(\kappa,\sq_\sigma)$ holds, then there exists a closed witness
  $c:[\kappa]^2\rightarrow\theta$ to $\U(\kappa,\kappa,\theta,\sup(\reg(\kappa)))$.
  Furthermore, if $\sigma=\aleph_0$ (i.e., if $\square(\kappa)$ holds), or if $\kappa$
  is ${<}\sigma$-inaccessible, then $\mathcal T(c)$ is a $\kappa$-Aronszajn tree.
\end{cor}

\begin{proof}
  By \cite[Theorem~1.24]{paper29}, $\square(\kappa,\sq_\sigma)$ entails
  a $C$-sequence $\langle C_\alpha\mid\alpha<\kappa\rangle$
  and a partition $\langle F_i\mid i<\kappa\rangle$ of $\kappa$ into fat sets
  such that, for every $i<\kappa$, every $\alpha\in F_i\cap E^\kappa_{\ge\sigma}$,
  and every $\bar\alpha\in\acc(C_\alpha)$, we have $\bar\alpha\in F_i$ and
  $C_{\bar\alpha}=C_\alpha\cap\bar\alpha$. For all $\alpha\in E^\kappa_{\ge\sigma}$,
  let $c_\alpha:=C_\alpha$. For all $\alpha\in E^\kappa_{<\sigma}$, let $c_\alpha$
  be some club in $\alpha$ of order type less than $\sigma$. For all $i<\theta$, let
  $H_i:=F_i\cap E^\kappa_{\ge\sigma}$. Now appeal to Theorem~\ref{simu} with
  $\vec C=\langle c_\alpha\mid\alpha<\kappa\rangle$ and $\langle H_i\mid i<\theta\rangle$
  to obtain the desired coloring, $c:[\kappa]^2\rightarrow \theta$.

  Finally, if $\sigma = \aleph_0$, or if $\kappa$ is ${<}\sigma$-inaccessible,
  then the hypothesis of Clause~(2) of Lemma~\ref{simu} holds.
  So, by Proposition~\ref{aronszajn}, in these cases, $\mathcal{T}(c)$
  is a $\kappa$-Aronszajn tree.
\end{proof}
\begin{remark} By \cite{lh_lucke}, $\square(\kappa)$ entails the principle
  $\square^{\ind}(\kappa,\theta)$ for every $\theta\in\reg(\kappa)$.
  In Part~III, we shall see that $\square^{\ind}(\kappa,\theta)$ yields the
  existence of a closed witness to $\U(\kappa,\kappa,\theta,\sup(\reg(\kappa)))$
  which is, moreover, subadditive.
\end{remark}

Since Corollary~\ref{l23} fully answers the question about the existence of closed witnesses to
$\U(\ldots)$ at successors of regular cardinals, we spend the remainder of this
section investigating, in turn, successors of singular cardinals and inaccessible
cardinals.

\subsection{Successors of singular cardinals} \label{singular_subsection}

We begin this subsection with an immediate corollary to Theorem~\ref{simu} that
can be seen as a counterpart to Corollary~\ref{l23}.

\begin{cor}\label{predecessor}
  Suppose that $\lambda$ is a singular cardinal. Then there exists a closed witness to
  $\U(\lambda^+,\lambda^+,\cf(\lambda),\lambda)$. Furthermore, if $\square_\lambda^*$ holds,
  then there exists a closed witness $c$ to $\U(\lambda^+,\lambda^+,\cf(\lambda),\lambda)$,
  for which $\mathcal T(c)$ is a $\lambda^+$-Aronszajn tree.
\end{cor}

\begin{proof}
  Set $\kappa:=\lambda^+$ and $\theta:=\cf(\lambda)$. Fix a $C$-sequence
  $\vec C=\langle C_\alpha\mid\alpha<\kappa\rangle$ with $\otp(C_\alpha) < \lambda$
  for all $\alpha<\kappa$. If $\square_\lambda^*$ holds,
  then we moreover require that $|\{ C_\alpha\cap\beta\mid \alpha<\kappa\}|\le\lambda$
  for all $\beta<\kappa$. Together with Proposition~\ref{aronszajn},
  this will ensure that, under $\square^*_\lambda$, the associated tree $\mathcal T(c)$
  will be $\lambda^+$-Aronszajn.

  Let $\langle \lambda_i\mid i<\theta\rangle$ be
  a strictly increasing sequence of infinite regular cardinals that converges to $\lambda$.
  For all $i < \theta$, let $H_i := E^\kappa_{\lambda_i}$.
  It is clear that, for every $\chi \in \reg(\lambda)$, we have
  $\sup\{ i<\theta\mid E^\kappa_{\ge\chi}\cap H_i \text{ is stationary}\}=\theta$.
  Moreover, for each $\alpha \in \acc(\kappa)$, there is $j < \theta$ such that
  $\otp(C_\alpha) < \lambda_j$, and hence $\sup\{i < \theta \mid \acc(C_\alpha)
  \cap H_i \neq \emptyset\} \leq j < \theta$, so we may appeal to
  Theorem~\ref{simu} to obtain a coloring $c$ as desired.
\end{proof}
\begin{remark} If $\lambda$ is a singular cardinal, $\square_\lambda$ holds and $2^\lambda=\lambda^+$,
then for every $\theta\in\reg(\lambda)$, there exists a closed witness $c$ to $\U(\lambda^+,\lambda^+,\theta,\lambda)$
for which $\mathcal T(c)$ is a nonspecial $\lambda^+$-Aronszajn tree.
The proof follows the arguments of the proofs of Corollaries \ref{l24} and \ref{vel},
building on Theorems 5.1, 5.3 and 1.24 of \cite{paper29}.
\end{remark}

Note that Theorem~\ref{t310} provides some limits on the extent of positive
$\zfc$ results regarding the existence of closed witnesses to $\U(\ldots)$
at successors of singular cardinals.
The rest of this subsection is devoted to obtaining positive
results under additional assumptions about the cardinals under consideration.
Note first that, by Corollary~\ref{l23}, if $\lambda$ is a singular cardinal,
$\theta, \chi \in \reg(\lambda)$, and there is a stationary $S \subseteq
E^{\lambda^+}_{\geq \chi}$ such that $\Tr(S) \cap E^{\lambda^+}_{\geq \theta}$
is nonstationary, then there is a closed witness to $\U(\lambda^+, \lambda^+,
\theta, \chi)$. The following theorem provides an improvement to this observation
by weakening the hypotheses.

\begin{thm}\label{l23s}
  Suppose that $\lambda$ is a singular cardinal, $\theta, \chi \in \reg(\lambda)$,
  and there exists a stationary $S\s E^{\lambda^+}_{\ge\chi}$ such that
  $$
    \sup\{ \nu<\lambda\mid  \Tr(S)\cap E^{\lambda^+}_\nu\text{ is stationary}\}<\lambda.
  $$
  Then there exists a closed witness to $\U(\lambda^+,\lambda^+,\theta,\chi)$.
\end{thm}

\begin{proof}
  By Corollary~\ref{l23}, we may assume that every stationary subset of
  $E^{\lambda^+}_{\ge\chi}$ reflects. We start by showing that we can find
  a stationary set as in the statement of the theorem that is slightly better-behaved.
  In particular, we find a stationary set that concentrates on a cofinality
  different from both $\aleph_0$ and $\cf(\lambda)$.

  \begin{claim}
    There exist $\sigma\in\reg(\lambda)\setminus\chi$ with
    $\sigma\notin\{\aleph_0,\cf(\lambda)\}$ and a stationary $\Delta\s E^{\lambda^+}_\sigma$,
    for which $\sup\{ \cf(\tau)\mid \tau<\lambda^+, \Delta\cap\tau\text{ is stationary}\}<\lambda$.
  \end{claim}

  \begin{proof}
    By the hypothesis of the theorem, we may fix a stationary $S_0\s E^{\lambda^+}_{\ge\chi}$
    and a club $D\s\lambda^+$ such that
    $$
      \sup\{ \cf(\tau)\mid \tau\in D\ \&\ S_0\cap\tau\text{ is stationary}\}<\lambda.
    $$
    Using Fodor's Lemma, and by shrinking $S_0$ if necessary, we may
    assume that $S_0 \subseteq D$ and there is $\sigma_0 \in \reg(\lambda)$
    such that $S_0 \subseteq E^{\lambda^+}_{\sigma_0}$.
    It is clear that $\sigma_0 \geq \chi$, so, by assumption, $\Tr(S_0)$
    is stationary. By another application of Fodor's Lemma, we may find
    a stationary set $S_1$ for which there is $\sigma_1 \in \reg(\lambda)$ such that
    $S_1 \subseteq \Tr(S_0) \cap E^{\lambda^+}_{\sigma_1}$. Clearly,
    $\Tr(S_1) \subseteq \Tr(S_0)$ and $\sigma_1 > \sigma_0$. Doing this again,
    find a stationary set $S_2$ for which there is $\sigma_2 \in \reg(\lambda) $
    such that $S_2 \subseteq \Tr(S_1) \cap E^{\lambda^+}_{\sigma_2}$.
    Altogether $\Tr(S_2) \subseteq \Tr(S_1)\s \Tr(S_0)\s D$ and $\sigma_2 > \sigma_1 > \sigma_0$. There must be
    $i < 3$ for which $\sigma_i \notin \{\aleph_0, \cf(\lambda)\}$. Choose such
    an $i$, and note that $\sigma := \sigma_i$ and $\Delta := S_i$ are as sought.
  \end{proof}

  Let $\sigma$ and $\Delta$ be given by the preceding claim, and fix
  $\mu\in\reg(\lambda)$ such that
  $$
    \max\{\sigma,\sup\{ \cf(\tau)\mid \tau<\lambda^+\ \&\ \Delta\cap\tau\text{ is stationary}\}\}<\mu.
  $$
  Fix a function $h:\lambda^+\rightarrow\theta$ such that, for all $i<\theta$,
  $H_i:=\{ \gamma\in E^{\lambda^+}_\mu\mid h(\gamma)=i\}$ is stationary.

  \begin{claim}
    There exists a $C$-sequence $\langle e_\delta\mid \delta<\lambda^+\rangle$ such that
    \begin{itemize}
      \item for all $\delta\in\acc(\lambda^+)$, $\otp(e_\delta)=\cf(\delta)$;
      \item for all $\delta\in E^{\lambda^+}_{\ge\mu}$, we have $e_\delta\cap\Delta=\emptyset$;
      \item for every club $D\s\lambda^+$ and every $i<\theta$, there exists $\delta\in\Delta$
        such that $\sup(e_\delta\cap H_i\cap D)=\delta$.
    \end{itemize}
  \end{claim}

  \begin{proof}
    It is clear how to obtain $e_\delta$ for each $\delta\in\lambda^+\setminus\Delta$.
    Now, to deal with $\delta\in\Delta$, proceed as follows. Let $\langle S_i\mid i<\theta\rangle$
    be some partition of $\Delta$ into stationary sets. For each $i < \theta$, since
    $S_i$ is a stationary subset of $E^{\lambda^+}_{\neq\cf(\lambda)}$,
    \cite[Proposition~1.4]{paper07} implies that $\clubsuit^-(S_i)$ holds. Thus,
    as pointed out on top of page 145 of \cite{paper_s01}, there exists a sequence
    $\langle B_\delta\mid \delta\in S_i\rangle$, with $\sup(B_\delta)=\delta$ for
    each $\delta\in S_i$, such that for every club $D\s\lambda^+$, the set
    $\{\delta\in S_i\mid B_\delta\s D\cap H_i \}$
    is stationary. Now, for each $i<\theta$ and $\delta\in S_i$, pick a
    club $e_\delta$ in $\delta$ of order-type $\sigma$ with $\nacc(e_\delta)\s B_\delta$.
  \end{proof}

  For each $\alpha < \lambda^+$, define a sequence $\langle C^n_\alpha\mid n<\omega\rangle$
  by recursion on $n<\omega$ as follows:
  \begin{itemize}
    \item $C^0_\alpha:=e_\alpha$;
    \item $C^{n+1}_\alpha:=\cl(C^n_\alpha\cup\bigcup\{ e_\delta\mid \delta\in\acc(C^n_\alpha)\cap\Delta\})$.
  \end{itemize}
  Then, let $C_\alpha:=\cl(\bigcup_{n<\omega}C_\alpha^n)$.

  \begin{claim}\label{claim3143}
    All of the following hold.
    \begin{enumerate}
      \item $\vec C:=\langle C_\alpha\mid\alpha<\lambda^+\rangle$ is a $C$-sequence.
      \item For all $\alpha\in\acc(\lambda^+)$, $\otp(C_\alpha)=\cf(\alpha)$.
      \item For all $\alpha\in\acc(\lambda^+)$ and $\delta\in(\acc(C_\alpha)\cup\{\alpha\})\cap\Delta$,
        we have $e_\delta\s C_\alpha$.
      \item For all $\alpha\in E^{\lambda^+}_{\ge\mu}$, we have $C_\alpha\cap\Delta=\emptyset$.
    \end{enumerate}
  \end{claim}

  \begin{proof}
    It is easy to see that $C_\alpha=e_\alpha$ for all $\alpha\in
    E^{\lambda^+}_{\le\sigma}\cup E^{\lambda^+}_{\ge\mu}$.
    Now, for all $\alpha<\lambda^+$ with $\sigma<\cf(\alpha)<\mu$, we have
    $\otp(C_\alpha^n)=\cf(\alpha)\cdot\sigma=\cf(\alpha)$ for all $n<\omega$.
  \end{proof}

  We now perform walks along $\vec C$ and derive a closed coloring
  $c:[\lambda^+]^2\rightarrow\theta$ as in the proof of Theorem~\ref{simu}
  by letting, for all $\alpha < \beta < \lambda^+$,
  $$
    c(\alpha,\beta):=\max\{ h(\tau)\mid \tau\in\im(\tr(\alpha,\beta))\}.
  $$
  By the implication $(2)\implies(3)$ of Lemma~\ref{pumpclosed}, the following
  claim suffices to finish the proof of the theorem.

  \begin{claim}
    Suppose that $\mathcal A\s[\lambda^+]^{<\chi}$ is a family consisting of
    $\lambda^+$-many pairwise disjoint sets, $D$ is a club in $\lambda^+$, and $i<\theta$.
    Then there exist $\gamma\in D$, $a\in\mathcal A$, and $\epsilon < \gamma$ such that
    \begin{itemize}
      \item $\gamma < a$;
      \item for all $\alpha \in (\epsilon, \gamma)$ and all $\beta\in a$, we have $c(\alpha,\beta)>i$.
    \end{itemize}
  \end{claim}

  \begin{proof} \renewcommand{\qedsymbol}{\ensuremath{\square \ \square}}
    Fix $\delta\in\Delta$ such that $\sup(e_\delta\cap H_{i+1}\cap D)=\delta$,
    and then fix an arbitrary $a\in\mathcal A$ with $\delta < a$. Let
    $\Lambda:=\sup\{\lambda_2(\delta,\beta)\mid\beta\in a\}$.
    As $\cf(\delta)=\sigma\ge\chi>|a|$, we have $\Lambda<\delta$, so we may pick
    $\gamma\in e_\delta\cap H_{i+1}\cap D$ above $\Lambda$.

    Let $T := \{\delta\} \cup \bigcup_{\beta \in a} \im(\tr(\delta, \beta))$, and let
    $$
      \epsilon := \sup\{\Lambda, ~ \sup(C_\tau \cap \gamma) \mid \tau
      \in T \ \&\ \sup(C_\tau \cap \gamma) < \gamma\}.
    $$
    Since $\cf(\gamma) = \mu > \chi > |a|$, we have $\epsilon < \gamma$.
    We claim that $\gamma$, $a$, and $\epsilon$ are as desired. To verify this,
    let $\alpha \in (\epsilon, \gamma)$ and $\beta \in a$ be arbitrary.
    We will show that $\gamma \in \tr(\alpha, \beta)$, and hence $c(\alpha, \beta)
    \geq h(\gamma) > i$. Let $\ell:=\rho_2(\delta,\beta)$.
    By Lemma~\ref{lambda2}, we have $\tr(\delta,\beta)\sq\tr(\alpha,\beta)$,
    and there are two cases to consider.

    $\br$ If $\delta\in\nacc(C_{\tr(\delta,\beta)(\ell-1)})$, then, since
    $\tr(\alpha,\beta)(\ell-1)=\tr(\delta,\beta)(\ell-1)$ and
    $$\sup(C_{\tr(\delta,\beta)(\ell-1)}\cap\delta) \leq\lambda_2(\delta,\beta)<\alpha<\gamma<\delta,$$
    we have $\tr(\alpha,\beta)(\ell)=\min(C_{\tr(\delta,\beta)(\ell-1)}\setminus\alpha)=\delta$.
    As $\otp(C_\delta)=\cf(\delta)=\sigma<\mu=\cf(\gamma)$,
    we have $\sup(C_\delta\cap\gamma)<\gamma$.
    As $\delta\in T$, we then have
    $\sup(C_\delta\cap\gamma)\le\epsilon<\alpha<\gamma$.
    Finally, since $\gamma\in e_\delta\s C_\delta$,
    we have $\tr(\alpha,\beta)(\ell+1)=\min(C_\delta\setminus\alpha)=\gamma$.

    $\br$ If $\delta\in\acc(C_{\tr(\delta,\beta)(\ell-1)})$, then, by Claim~\ref{claim3143}(4),
    it follows that $\cf(\tr(\delta,\beta)(\ell-1))<\mu=\cf(\gamma)$, and hence
    $\sup(C_{\tr(\delta,\beta)(\ell-1)}\cap\gamma)\le\epsilon<\alpha<\gamma$.
    Consequently, $\tr(\alpha,\beta)(\ell)=\min(C_{\tr(\delta,\beta)(\ell-1)}\setminus\alpha)=\gamma$.
  \end{proof}
  \let\qed\relax
\end{proof}

We next show that the existence of closed witnesses to $\U(\lambda^+, \lambda^+,
\theta, \cf(\lambda))$ follows from a local instance of $\gch$.

\begin{thm}\label{ch}
  Suppose that $\lambda$ is a singular cardinal, $\theta \in \reg(\lambda)$,
  and $2^\lambda=\lambda^+$. Then there exists a closed witness to
  $\U(\lambda^+,\lambda^+,\theta,\cf(\lambda))$.
\end{thm}

\begin{proof}
  Let $\chi:=\max\{\theta,\cf(\lambda)\}^+$ and $\Delta:=E^{\lambda^+}_\chi$.
  As $\Delta$ is a stationary subset of $E^{\lambda^+}_{\neq\cf(\lambda)}$ and $2^\lambda=\lambda^+$,
  \cite[Claim~2.3]{Sh:922} provides us with a $\diamondsuit(\Delta)$-sequence,
  $\langle X_\delta\mid\delta\in \Delta\rangle$.

  Let $\langle \lambda_j\mid j<\cf(\lambda)\rangle$ be a strictly increasing
  sequence of regular cardinals converging to $\lambda$, with $\lambda_0>\chi$.
  Fix a function $h:\lambda^+\rightarrow\theta$ such that for every $i<\theta$ and
  $j<\cf(\lambda)$, the following set is stationary:
  $$
    H^i_j:=\{ \gamma\in E^{\lambda^+}_{\lambda_j}\mid h(\gamma)=i\}.
  $$

  Now, let $\langle e_\delta\mid\delta<\lambda^+\rangle$ be a $C$-sequence such that
  \begin{itemize}
    \item for all $\delta\in\acc(\lambda^+)$, $\otp(e_\delta)=\cf(\delta)$;
    \item for all $\delta\in \Delta$, $i<\theta$, and $j<\cf(\lambda)$, if
      $\sup(X_\delta\cap H^i_j)=\delta$, then $\sup(X_\delta\cap H^i_j\cap e_\delta)=\delta$.
  \end{itemize}

  For each $\alpha < \lambda^+$, define a sequence $\langle C^n_\alpha\mid n<\omega\rangle$
  by recursion on $n<\omega$ as follows:
  \begin{itemize}
    \item $C^0_\alpha:=e_\alpha$;
    \item $C^{n+1}_\alpha:=\cl(C^n_\alpha\cup\bigcup\{ e_\delta\mid \delta\in\acc(C^n_\alpha)\cap\Delta\})$.
  \end{itemize}
  Then, let $C_\alpha:=\cl(\bigcup_{n<\omega}C_\alpha^n)$.

  \begin{claim} \label{4171}
    All of the following hold.
    \begin{enumerate}
      \item $\vec C:=\langle C_\alpha\mid\alpha<\lambda^+\rangle$ is a $C$-sequence.
      \item For all $\alpha\in\acc(\lambda^+)$, $\otp(C_\alpha)<\lambda$.
      \item For all $\alpha\in\acc(\lambda^+)$ and $\delta\in(\acc(C_\alpha)\cup\{\alpha\})\cap\Delta$,
        $e_\delta\s C_\alpha$.
      \item For every club $D$ in $\lambda^+$, there is $\delta\in \Delta$ such that
        for every $\mu<\lambda$, $\Lambda<\delta$, and $i<\theta$, there is
        $\gamma\in e_\delta\cap D$ such that $\cf(\gamma)>\mu$, $\gamma>\Lambda$ and $h(\gamma)=i$.
    \end{enumerate}
  \end{claim}

  \begin{proof}
    Clause~(1) is straightforward.

    (2) Fix $\alpha \in \acc(\lambda^+)$. $|C^0_\alpha| = \cf(\alpha)$
    and, by induction on $n$, it is then easy to see that, for all $n < \omega$,
    we have $|C^{n+1}_\alpha| \leq \cf(\alpha) \cdot \chi$. It follows that
    $|C_\alpha| \leq \cf(\alpha) \cdot \chi < \lambda$.

    (3) For each $\alpha\in\Delta$, since $\otp(e_\alpha)=\cf(\alpha)=\chi$,
    we simply have $\acc(C^n_\alpha)\cap\Delta=\emptyset$ for all $n<\omega$, so $C_\alpha=e_\alpha$.
    Now, for $\alpha\in\acc(\lambda^+)$ and $\delta\in\acc(C_\alpha)$,
    since $\cf(\delta)=\chi>\omega$, there must exist some $n<\omega$
    such that $\delta\in\acc(C^n_\alpha)$, and hence $e_\delta\s C_\alpha^{n+1}\s C_\alpha$.

    (4) Fix a club $D$ in $\lambda^+$. For each $i < \theta$ and $j<\cf(\lambda)$,
    $H^i_j$ is stationary, so
    $$
      E:=\bigcap_{i<\theta}\bigcap_{j<\cf(\lambda)}\acc^+(D\cap H^i_j)
    $$
    is a club. Since $\{ \delta\in \Delta\mid D\cap\delta=X_\delta\}$ is stationary,
    we can pick $\delta\in \Delta\cap E$ such that $D\cap\delta=X_\delta$. For all
    $i < \theta$ and $j < \cf(\lambda)$, we have $\sup(X_\delta\cap H^i_j)=\delta$,
    and hence $\sup(X_\delta\cap H^i_j\cap e_\delta)=\delta$. In particular,
    for every $\mu<\lambda$, $\Lambda<\delta$, and $i<\theta$, we may fix some
    $j<\cf(\lambda)$ such that $\lambda_j>\mu$ and then find $\gamma\in X_\delta\cap
    H^i_j\cap e_\delta$ above $\Lambda$. Clearly, $\cf(\gamma)>\mu$,
    $h(\gamma)=i$, and $\gamma$ is an accumulation point of the club $D$
  \end{proof}

  We now perform walks along $\vec C$ and derive a closed coloring
  $c:[\kappa]^2\rightarrow\theta$ as in the proof of Theorem~\ref{simu}
  by letting, for all $\alpha < \beta < \lambda^+$,
  $$
    c(\alpha,\beta):=\max\{ h(\tau)\mid \tau\in\im(\tr(\alpha,\beta))\}.
  $$
  We claim that $c$ witnesses $\U(\lambda^+, \lambda^+, \theta, \cf(\lambda))$
  and prove this by verifying Clause~(2) of Lemma~\ref{pumpclosed}. To this
  end, fix a family $\mathcal{A} \subseteq [\lambda^+]^{<\cf(\lambda)}$
  consisting of $\lambda^+$-many pairwise disjoint sets, a club $D$ in $\lambda^+$,
  and a color $i < \theta$. We will find $\gamma \in D$, $a \in \mathcal{A}$,
  and $\epsilon < \gamma$ such that
  \begin{itemize}
    \item $\gamma < a$;
    \item for all $\alpha \in (\epsilon, \gamma)$ and all $\beta \in a$, we have
      $c(\alpha, \beta) > i$.
  \end{itemize}

  Use Clause~(4) of Claim~\ref{4171} to find $\delta\in \Delta$ such that,
  for every $\mu<\lambda$ and $\Lambda<\delta$, there exists $\gamma\in e_\delta\cap D$
  such that $\cf(\gamma)>\mu$, $\gamma>\Lambda$, and $h(\gamma)=i+1$. Fix an
  arbitrary $a\in\mathcal A$ with $\delta < a$, and set
  \begin{itemize}
    \item $\Lambda:=\sup\{\lambda_2(\delta,\beta)\mid \beta\in a\}$;
    \item $C:=C_\delta \cup \bigcup_{\beta\in a}\bigcup_{\tau\in\im(\tr(\delta,\beta))}C_\tau$.
  \end{itemize}
  As $|a|<\cf(\lambda)<\cf(\delta)$, we have $\Lambda<\delta$ and $|C|<\lambda$.
  Thus, we can pick $\gamma\in e_\delta \cap D$ such that $\cf(\gamma)>|C|$,
  $\gamma>\Lambda$, and $h(\gamma)=i+1$. Let $\epsilon := \max\{\Lambda, \sup(C \cap \gamma)\}$.

  We claim that $\gamma$, $a$, and $\epsilon$ are as desired. To this end, let
  $\alpha \in (\epsilon, \gamma)$ and $\beta \in a$ be arbitrary. Then
  $$
    \lambda_2(\delta,\beta)\le\Lambda\le\epsilon<\alpha<\gamma<\delta<\beta,
  $$
  so, by Lemma~\ref{lambda2}, $\tr(\delta,\beta)\sq \tr(\alpha,\beta)$.
  We claim that $\gamma \in \im(\tr(\alpha, \beta))$. To see this, let
  $\ell:=\rho_2(\delta,\beta)$, and consider the following two cases, each
  of which will involve the use of Clause~(3) of Claim~\ref{4171}.

  $\br$ If $\delta\in\nacc(C_{\tr(\delta,\beta)(\ell-1)})$, then
  $\sup(C_{\tr(\delta,\beta)(\ell-1)}\cap\delta) \leq\lambda_2(\delta,\beta)$, and hence
  $$
    [\alpha,\delta)\cap C_{\tr(\delta,\beta)(\ell-1)}\s(\lambda_2(\delta,\beta),\delta)
    \cap C_{\tr(\delta,\beta)(\ell-1)}=\emptyset.
  $$
  Consequently, $\tr(\alpha,\beta)(\ell)=\min(C_{\tr(\alpha,\beta)(\ell-1)}\setminus\alpha)
  =\min(C_{\tr(\delta,\beta)(\ell-1)}\setminus\alpha)=\delta$. As
  $$
    [\alpha,\gamma)\cap C_{\tr(\alpha,\beta)(\ell)}\s (\epsilon,\gamma)\cap C_\delta=\emptyset
  $$
  and $\gamma\in e_{\delta}\s C_{\delta}\s C$, we have $\tr(\alpha,\beta)(\ell+1)=
  \min(C_\delta\setminus\alpha)=\gamma$.

  $\br$ If $\delta\in\acc(C_{\tr(\delta,\beta)(\ell-1)})$, then
  $$
    [\alpha,\gamma)\cap  C_{\tr(\delta,\beta)(\ell-1)}\s(\epsilon,\gamma)\cap C=\emptyset.
  $$
  As $\delta\in\acc(C_{\tr(\delta,\beta)(\ell-1)}) \cap \Delta$, we have
  $\gamma\in e_{\delta}\s C_{\tr(\delta,\beta)(\ell-1)}$, and hence
  $\tr(\alpha,\beta)(\ell)=\min(C_{\tr(\delta,\beta)(\ell-1)}\setminus\alpha)=\gamma$.

  In either case, we have shown that $\gamma\in\im(\tr(\alpha, \beta))$,
  and hence $c(\alpha,\beta)\ge h(\gamma)=i+1$.
\end{proof}

Our final result of this subsection shows that a failure of the simultaneous
stationary reflection principle $\refl({<}\cf(\lambda), \lambda^+)$ entails the
existence of a closed witness to $\U(\lambda^+, \lambda^+, \theta, \cf(\lambda))$
for all $\theta \in \reg(\lambda)$.

\begin{thm}\label{lemma215}\label{increasechi}
  Suppose that $\lambda$ is a singular cardinal and $\theta \in \reg(\lambda)$.
  If any one of the following conditions holds:
  \begin{enumerate}
    \item $\refl({<}\cf(\lambda),\lambda^+)$ fails;
    \item $\cf(\ns_{\cf(\lambda)},\s)<\lambda$ and $\theta<\cf(\lambda)$;
    \item there exists a tail-closed witness to $\U(\lambda^+,2,\theta,2)$;
    \item there exists a somewhere-closed witness to $\U(\lambda^+,2,\theta,\omega)$;
  \end{enumerate}
  then there exists a closed witness to $\U(\lambda^+,\lambda^+,\theta,\cf(\lambda))$.
\end{thm}

  The rest of this subsection is dedicated to the proof of Theorem~\ref{increasechi}.
  The proof splits into two cases based on whether $\lambda$ has uncountable or
  countable cofinality. The structures of the proofs in the two cases are similar
  to one another. We begin by identifying a useful club-guessing sequence (or, in the countable
  cofinality case, an ``off-center" club-guessing matrix) and its associated ideal.
  We use these objects to identify a $C$-sequence (or, in the countable cofinality
  case, a collection of $C$-sequences) along which we will perform walks. After
  isolating the salient properties of walks along these $C$-sequences, we
  will verify, in turn, that each of the conditions identified in the statement
  of the theorem implies the existence of a closed witness to $\U(\lambda^+,\lambda^+,\theta,\cf(\lambda))$.
  Let us begin.

  \medskip

  \textbf{Case 1: Uncountable cofinality.} Assume in this case that $\cf(\lambda) > \omega$.
  By \cite[Theorem~2]{EiSh:819}, we may find a stationary $\Delta\s E^{\lambda^+}_{\cf(\lambda)}$
  and a sequence $\vec e=\langle e_\delta\mid \delta\in \Delta\rangle$ such that
  \begin{itemize}
    \item for every $\delta\in \Delta$, $e_\delta$ is a club in $\delta$ of order type $\cf(\lambda)$;
    \item for every $\delta\in \Delta$, $\langle \cf(\gamma)\mid \gamma\in \nacc(e_\delta)\rangle$
      is strictly increasing and converging to $\lambda$;
    \item for every club $D$ in $\lambda^+$, there exists $\delta\in \Delta$ such that $e_\delta\s D$.
  \end{itemize}

  Now, define $\mathcal I\s\mathcal P(\lambda^+)$ by letting $A\in\mathcal P(\lambda^+)$
  be in $\mathcal I$ iff there exists a club $D\s\lambda^+$ such that for every
  $\delta\in \Delta\cap D$, we have $\sup(\nacc(e_\delta)\cap D\cap A)<\delta$.

  \begin{claim}\label{3181}
    $\mathcal{I}$ satisfies all of the following properties:
    \begin{enumerate}
      \item[(a)] $\mathcal I$ is a $\cf(\lambda)$-complete proper ideal over $\lambda^+$, extending $\ns_{\lambda^+}$;
      \item[(b)] $\mathcal I$ is $\tau$-indecomposable for all $\tau\in\reg(\lambda)\setminus\{\cf(\lambda)\}$;
      \item[(c)] if $\cf(\ns_{\cf(\lambda)},\s)<\lambda$,
      then $\mathcal I$ is not weakly $\cf(\lambda)$-saturated;
      \item[(d)] for all $\sigma<\lambda$, $E^{\lambda^+}_{\ge\sigma}\in\mathcal I^*$.
    \end{enumerate}
  \end{claim}

  \begin{proof}
    (a)  It is clear that $\mathcal I$ is downward closed and contains all
    nonstationary subsets of $\lambda^+$. Also, by the choice of $\vec e$, we know
    that $\lambda^+\notin\mathcal I$. Finally, since $\cf(\delta)=\cf(\lambda)$ for
    all $\delta\in\Delta$, and since the intersection of fewer than $\cf(\lambda)$-many clubs
    in $\lambda^+$ is a club, we infer that $\mathcal I$ is $\cf(\lambda)$-complete.

    (b) Suppose that $\tau\in\reg(\lambda)\setminus\{\cf(\lambda)\}$ and that
    $\vec A=\langle A_j\mid j<\tau\rangle$ is a $\s$-increasing sequence of elements
    from $\mathcal I$. We shall show that $A:=\bigcup_{j<\tau}A_j$ is in $\mathcal I$, as well.
    For each $j<\tau$, pick a witnessing club $D_j$. We claim that the club
    $D:=\bigcap_{j<\tau}D_j$ witnesses that $A\in\mathcal I$. To see this, let
    $\delta\in\Delta\cap D$ be arbitrary. Then $\sup(\nacc(e_\delta)\cap D\cap A_j)<\delta$
    for all $j<\tau$. As $\cf(\delta)\neq\cf(\tau)$ and $\vec A$ is $\s$-increasing, we
    infer that $\sup(\nacc(e_\delta)\cap D\cap A)<\delta$, as well.

    (c) Using the fact that $\cf(\ns_{\cf(\lambda)},\s)<\lambda$, fix a sequence
    $\langle C^\iota\mid \iota<\lambda\rangle$ of clubs in $\cf(\lambda)$ such that,
    for every club $C$ in $\cf(\lambda)$, there exists $\iota<\lambda$ with $C^\iota\s C$.
    For each $\iota<\lambda$ and $j<\cf(\lambda)$, we let $C^\iota(j)$ denote the
    unique $\alpha\in C^\iota$ with $\otp(C^\iota\cap\alpha)=j$. Let
    $\langle \lambda_j\mid j<\cf(\lambda)\rangle$ be a strictly increasing and
    continuous sequence of cardinals converging to $\lambda$. For every $\iota<\lambda$,
    define $h^\iota:\lambda^+\rightarrow\cf(\lambda)$ by setting, for all $\gamma < \lambda^+$,
    $$
      h^\iota(\gamma):=\min\{ j<\cf(\lambda)\mid \cf(\gamma)\le\lambda_{C^\iota(j)}\}.
    $$
    Fix a surjection $\varphi:\cf(\lambda)\rightarrow\cf(\lambda)$ such that
    $|\varphi^{-1}\{i\} \cap \nacc(\cf(\lambda))|=\cf(\lambda)$ for all $i<\cf(\lambda)$,
    and then let $\Gamma^\iota_i:=\{ \gamma<\lambda^+\mid \varphi(h^\iota(\gamma))=i\}.$

    We claim that there is $\iota < \lambda$ for which $\langle \Gamma^\iota_i \mid
    i < \cf(\lambda) \rangle$ is a counterexample to the weak $\cf(\lambda)$-saturation of $\mathcal{I}$.
    It is trivial to see that, for all $\iota<\lambda$, $\langle \Gamma^\iota_i\mid
    i<\cf(\lambda)\rangle$ is a partition of $\lambda^+$. Thus, it suffices to prove
    that there exists some $\iota<\lambda$ such that, for all $i<\cf(\lambda)$,
    $\Gamma^\iota_i\in\mathcal I^+$. Suppose that this is not the case, and, for each
    $\iota<\lambda$, fix a club $D^\iota\s\lambda^+$ and an $i(\iota)<\cf(\lambda)$
    such that, for every $\delta\in\Delta\cap D^\iota$, we have $\sup(\nacc(e_\delta)\cap
    D^\iota\cap \Gamma^\iota_{i(\iota)})<\delta$. Consider the club
    $D:=\bigcap_{\iota<\lambda}D^\iota$. Pick $\delta\in \Delta$ such that $e_\delta\s D$.
    Let
    $$
      C:=\{ j<\cf(\lambda)\mid \lambda_j\in\acc^+(\{\cf(\gamma)\mid
      \gamma\in\nacc(e_\delta)\})\},
    $$
    and note that $C$ is a club in $\cf(\lambda)$.
    Find $\iota<\lambda$ such that $C^\iota\s C$. For each $j<\cf(\lambda)$, as
    $C^\iota(j+1)\in C$, we have $\lambda_{C^{\iota}(j+1)}\in\acc^+(\{\cf(\gamma)\mid
    \gamma\in\nacc(e_\delta)\})$, so there exists some $\gamma\in\nacc(e_\delta)$
    such that $h^\iota(\gamma)=j+1$. Thus, $h^\iota[\nacc(e_\delta)] \supseteq
    \nacc(\cf(\lambda)) \setminus \{0\}$, so, by the choice of $\varphi$, it follows that,
    for all $i<\cf(\lambda)$,
    $$
      |\{\gamma\in\nacc(e_\delta)\mid \varphi(h^\iota(\gamma))=i\}|=\cf(\lambda)=\otp(e_\delta).
    $$
    In particular, $\sup(\nacc(e_\delta)\cap \Gamma^\iota_{i(\iota)})=\delta$,
    contradicting the fact that $e_\delta\s D\s D^\iota$.

    (d) By the choice of $\vec e$, we have $E^{\lambda^+}_{<\sigma}\in\mathcal I$ for all $\sigma<\lambda$.
  \end{proof}

  Next, by a standard club-swallowing trick (see the procedure before Claim~\ref{claim3143}), we may find a $C$-sequence
  $\vec C=\langle C_\alpha\mid\alpha<\lambda^+\rangle$ such that
  \begin{itemize}
    \item for all $\alpha\in\acc(\lambda^+)$, $C_\alpha$ is a club in $\alpha$ of
      order-type $<\lambda$;
    \item for all $\alpha\in\acc(\lambda^+)$ and $\delta\in(\acc(C_\alpha)\cup\{\alpha\})
      \cap\Delta$, we have $e_\delta\s C_\alpha$.
  \end{itemize}
  Let $\tr:[\lambda^+]^2\rightarrow{}^{<\omega}\lambda^+$ denote the function derived from walking along $\vec{C}$.

  \begin{claim}\label{2122}
    Suppose that $\mathcal A\s[\lambda^+]^{<\cf(\lambda)}$ is a family consisting
    of $\lambda^+$-many pairwise disjoint sets, $D$ is a club in $\lambda^+$, and
    $\Gamma\in \mathcal{I}^+$. Then there exist
    $\gamma\in D\cap\Gamma$, $a \in \mathcal{A}$, and $\epsilon < \gamma$ such that
    \begin{itemize}
      \item $\gamma < a$;
      \item for all $\alpha\in(\epsilon,\gamma)$ and all $\beta \in a$,
        we have $\gamma\in\im(\tr(\alpha,\beta))$.
    \end{itemize}
  \end{claim}

\begin{proof} As $\Gamma\notin\mathcal I$,
let us fix some $\delta\in\Delta$ such that $\sup(\nacc(e_\delta)\cap D\cap\Gamma)=\delta$.
Pick an arbitrary $a\in\mathcal A$ with $a>\delta$, and put
\begin{itemize}
\item $\Lambda:=\sup\{\lambda_2(\delta,\beta)\mid \beta\in a\}$; and
\item $C:=C_\delta\cup\bigcup_{\beta\in a}\bigcup_{\tau\in\im(\tr(\delta,\beta))}C_\tau$.
\end{itemize}
As $|a|<\cf(\lambda)=\cf(\delta)$, we have $\Lambda<\delta$ and $|C|<\lambda$.
Pick $\gamma\in\nacc(e_\delta)\cap D\cap\Gamma$ such that $\gamma>\Lambda$ and $\cf(\gamma)>|C|$.
Let $\epsilon:=\max\{\Lambda,\sup(C\cap\gamma)\}$. As $\cf(\gamma)>|C|$, we have $\epsilon<\gamma$.
We shall show that $\gamma$, $a$ and $\epsilon$ are as sought.

To this end, fix arbitrary $\alpha\in(\epsilon,\gamma)$ and $\beta\in a$.
We have $$\lambda_2(\delta,\beta)\le\Lambda\le\epsilon<\alpha<\gamma<\delta<\beta,$$
so, by Lemma~\ref{lambda2}, $\tr(\delta,\beta)\sq \tr(\alpha,\beta)$.
Let $\ell:=\rho_2(\delta,\beta)$.
There are now two cases to consider.

$\br$ If $\delta\in\nacc(C_{\tr(\delta,\beta)(\ell-1)})$,
then $\sup(C_{\tr(\delta, \beta)(\ell-1)} \cap \delta) \leq \lambda_2(\delta,\beta)$ and hence
$$[\alpha,\delta)\cap C_{\tr(\delta,\beta)(\ell-1)}\s(\lambda_2(\delta,\beta),\delta)\cap C_{\tr(\delta,\beta)(\ell-1)}=\emptyset.$$
Consequently, $\tr(\alpha,\beta)(\ell)=\min(C_{\tr(\alpha,\beta)(\ell-1)}\setminus\alpha)=\min(C_{\tr(\delta,\beta)(\ell-1)}\setminus\alpha)=\delta$.
As $$[\alpha,\gamma)\cap C_{\tr(\alpha,\beta)(\ell)}\s (\epsilon,\gamma)\cap C=\emptyset$$
and $\gamma\in e_\delta\s C_\delta=C_{\tr(\alpha,\beta)(\ell)}\s C$,
we have $\tr(\alpha,\beta)(\ell+1)=\min(C_{\tr(\alpha,\beta)(\ell)}\setminus\alpha)=\gamma$.

$\br$ If $\delta\in\acc(C_{\tr(\delta,\beta)(\ell-1)})$, then
$$[\alpha,\gamma)\cap  C_{\tr(\delta,\beta)(\ell-1)}\s(\epsilon,\gamma)\cap C=\emptyset.$$
As $\delta\in\acc(C_{\tr(\delta,\beta)(\ell-1)})\cap\Delta$,
we have $\gamma\in e_\delta\s C_{\tr(\delta,\beta)(\ell-1)}\s C$.
It follows that $\tr(\alpha,\beta)(\ell)=\min(C_{\tr(\delta,\beta)(\ell-1)}\setminus\alpha)=\gamma$.
\end{proof}

  We are now ready to begin verifying, in turn, that the existence of a closed witness to
  $\U(\lambda^+, \lambda^+, \theta, \cf(\lambda))$ follows from each of the
  conditions isolated in the statement of the theorem. We begin with condition (2).

  \begin{claim}\label{3183}
    Suppose that $\mathcal I$ is not weakly $\theta$-saturated. Then there exists
    a closed witness to $\U(\lambda^+,\lambda^+,\theta,\cf(\lambda))$.
  \end{claim}

  \begin{proof}
    Fix a function $h:\lambda^+\rightarrow\theta$ such that, for all $i<\theta$,
    $h^{-1}\{i\}\in I^+$. Derive a closed coloring $c:[\lambda^+]^2\rightarrow\theta$
    as in the proof of Theorem~\ref{simu} by letting, for all $\alpha < \beta < \lambda^+$,
    $$c
      (\alpha,\beta):=\max\{ h(\tau)\mid \tau\in\im(\tr(\alpha,\beta))\}.
    $$
    To show that $c$ witnesses $\U(\lambda^+,\lambda^+,\theta,\cf(\lambda))$,
    it suffices to verify Clause~(2) of Lemma~\ref{pumpclosed}. To this end,
    fix a family $\mathcal A\s[\lambda^+]^{<\cf(\lambda)}$ consisting of
    $\lambda^+$-many pairwise disjoints sets, a club $D\s\lambda^+$, and a color $i<\theta$.
    As $h^{-1}\{i+1\}\in\mathcal I^+$ and $E^{\lambda^+}_{\ge\cf(\lambda)}\in\mathcal I^*$,
    Claim~\ref{2122} provides us with $\gamma\in D\cap h^{-1}\{i+1\}\cap
    E^{\lambda^+}_{\ge\cf(\lambda)}$, $a \in \mathcal{A}$, and $\epsilon < \gamma$
    such that $\gamma < a$ and, for all $\alpha\in(\epsilon,\gamma)$ and all $\beta \in a$,
    we have $\gamma\in\im(\tr(\alpha,\beta))$. Then $c(\alpha,\beta)\ge h(\gamma)>i$
    for all $\alpha\in(\epsilon,\gamma)$ and all $\beta\in a$, so $\gamma$, $a$,
    and $\epsilon$ witness the conclusion of Clause~(2) of Lemma~\ref{pumpclosed}.
  \end{proof}

  In particular, it follows from Claim~\ref{3181}(3) that, if
  $\cf(\ns_{\cf(\lambda)},\s)<\lambda$ and $\theta\le\cf(\lambda)$, then there
  exists a closed witness to $\U(\lambda^+,\lambda^+,\theta,\cf(\lambda))$.

  We now turn our attention to conditions (3) and (4) from the statement of the
  theorem, which are taken care of by the next claim.

  \begin{claim}\label{3184}
    Suppose that $c:[\lambda^+]^2\rightarrow\theta$ is a coloring,
    $\sigma\in\reg(\lambda)$, and one of the following conditions holds:
    \begin{itemize}
      \item $c$ is a somewhere-closed witness to $\U(\lambda^+,2,\theta,\omega)$; or
      \item $c$ is an $E^{\lambda^+}_{\ge\sigma}$-closed witness to $\U(\lambda^+,2,\theta,2)$.
    \end{itemize}
    Then there exists a closed witness to $\U(\lambda^+,\lambda^+,\theta,\cf(\lambda))$.
  \end{claim}

  \begin{proof}
    Define $d:[\lambda^+]^2\rightarrow\theta$ by setting, for all $\alpha < \beta < \lambda^+$,
    $$
      d(\alpha,\beta):=\max\{ c(\delta,\gamma)\mid (\delta,\gamma)\in[\im(\tr(\alpha,\beta))]^2\},
    $$
    provided that the set is nonempty, and $d(\alpha,\beta):=0$, otherwise.

    We claim that $d$ is as desired. To see that $d$ is closed, suppose that
    $\beta<\lambda^+$, $i < \theta$, and $A\s D^d_{\leq i}(\beta)$, with $\eta :=
    \sup(A)$ smaller than $\beta$. To show that $\eta \in D^d_{\leq i}(\beta)$, fix
    $\alpha\in A$ above $\lambda_2(\eta,\beta)$. By Lemma~\ref{lambda2},
    $\im(\tr(\eta,\beta))\s\im(\tr(\alpha,\beta))$, and hence, by the definition of $d$,
    we have $d(\eta,\beta)\le d(\alpha,\beta)\le i$.

    To see that $d$ witnesses $\U(\lambda^+,\lambda^+,\theta,\cf(\lambda))$,
    it suffices to verify Clause~(2) of Lemma~\ref{pumpclosed}.
    To this end, suppose that $\mathcal A\s[\lambda^+]^{<\cf(\lambda)}$ is a
    family consisting of $\lambda^+$-many pairwise disjoint sets, $D$ is a club in
    $\lambda^+$, and $i<\theta$. We shall prove that there exist $\zeta\in D$,
    $a\in\mathcal A$, and $\epsilon^* < \zeta$ for which
    \begin{itemize}
      \item $\zeta < a$;
      \item for all $\alpha \in (\epsilon^*, \zeta)$ and all $\beta\in a$,
        we have $d(\alpha,\beta)>i$.
    \end{itemize}
    Let $\Gamma$ be the set of $\gamma \in E^{\lambda^+}_{\geq \sigma}$ for which
    there exist $a \in \mathcal{A}$ and $\epsilon < \gamma$ such that
    \begin{itemize}
      \item $\gamma < a$;
      \item for all $\beta \in a$ and $\alpha \in (\epsilon, \gamma)$, we have
        $\gamma \in \im(\tr(\alpha, \beta))$.
    \end{itemize}
    By Claim~\ref{2122} and Claim~\ref{3181}(4), $\Gamma$ is stationary.
    For each $\gamma\in\Gamma$, pick $a_\gamma\in\mathcal A$ and $\epsilon_\gamma<\gamma$
    witnessing $\gamma\in\Gamma$. Fix a stationary subset $\Gamma'\s\Gamma$
    on which the map $\gamma\mapsto\epsilon_\gamma$ is constant, with value, say,
    $\epsilon$. Now, let $S$ be the set of $\delta < \lambda^+$ for which
    there exist $\gamma \in \Gamma'$ and $\varepsilon < \delta$ such that
    \begin{itemize}
      \item $\delta < \gamma$;
      \item for all $\zeta \in (\varepsilon, \delta)$, we have $c(\zeta, \gamma) > i$.
    \end{itemize}

    We claim that $S$ is stationary. To see this, consider the following two cases.

    $\br$ If $c$ is a somewhere-closed witness to $\U(\lambda^+,2,\theta,\omega)$,
    then by the implication $(1)\implies(2)$ of Lemma~\ref{pumpclosed},
    we infer that $S$ is stationary.

    $\br$ If $c$ is a $E^{\lambda^+}_{\ge\sigma}$-closed witness to $\U(\lambda^+,2,\theta,2)$,
    then repeating the proof of the implication $(1)\implies(2)$ of Lemma~\ref{pumpclosed}
    in the current setting implies that, furthermore, $S\cap\Gamma'$ is stationary.

    For each $\delta\in S$, pick $\gamma_\delta\in\Gamma'$ and $\varepsilon_\delta<\delta$
    witnessing $\delta\in S$. Fix a stationary subset $S'\s S$ on which the map
    $\delta\mapsto\varepsilon_\delta$ is constant, with value, say, $\varepsilon$.
    Finally, let $Z$ be the set of $\zeta \in E^{\lambda^+}_{\geq \cf(\lambda)}$
    for which there exist $\delta \in S'$ and $\eta < \zeta$ such that
    \begin{itemize}
      \item $\zeta < \delta$;
      \item for all $\alpha \in (\eta, \zeta)$, we have $\zeta \in \im(\tr(\alpha, \gamma_\delta))$.
    \end{itemize}
    By Claim~\ref{2122} and Claim~\ref{3181}(4), $Z$ is stationary, so we may
    find $\zeta\in Z\cap D$ above $\max\{\epsilon,\varepsilon\}$. Pick $\delta\in S'$
    and $\eta<\zeta$ above $\max\{\epsilon,\varepsilon\}$ such that $\delta > \zeta$ and, for all $\alpha\in(\eta,\zeta)$,
    we have $\zeta\in\im(\tr(\alpha,\gamma_\delta))$. Set $a:=a_{\gamma_\delta}$
    and $\epsilon^* := \eta$.
    We claim that $\zeta$, $a$, and $\epsilon^*$ are as desired. To this end,
    let $\alpha \in (\epsilon^*, \zeta)$ and $\beta \in a$ be arbitrary. Then
    $$
      \max\{\epsilon,\varepsilon\}<\eta<\alpha<\zeta<\delta<\gamma_{\delta}<\beta.
    $$
    As $\gamma_\delta \in \Gamma'$, $\beta\in a_{\gamma_\delta}$, and
    $\alpha\in(\epsilon, \gamma_\delta)$, we have $\gamma_{\delta}\in\im(\tr(\alpha,\beta))$.
    Next, as $\alpha\in(\eta,\zeta)$, we have $\zeta\in\im(\tr(\alpha,\gamma_\delta))$, and,
    consequently, $\zeta\in\im(\tr(\alpha,\beta))$. Finally, as $\delta \in S'$ and
    $\zeta\in(\varepsilon,\delta)$, we infer that $c(\zeta,\gamma_\delta)>i$.
    Altogether, we obtain $d(\alpha,\beta)>i$, as desired.
  \end{proof}

  We end the uncountable cofinality case of the proof by addressing condition (1).

  \begin{claim} Suppose that there exists no closed witness to
    $\U(\lambda^+,\lambda^+,\theta,\cf(\lambda))$. Then $\refl({<}\cf(\lambda),\lambda^+)$ holds.
  \end{claim}

  \begin{proof}
    By Corollary~\ref{predecessor}, $\theta\neq\cf(\lambda)$.
    So, by Claims \ref{3181} and \ref{3183}, it follows that
    $\mathcal I$ is a $\cf(\lambda)$-complete ideal which is weakly $\theta$-saturated
    and $\theta$-indecomposable. But then by
    \cite[Theorem~2(4)]{MR2652193}, $\refl({<}\cf(\lambda),S^*)$ holds for
    $S^*:=E^{\lambda^+}_{\ge\theta}\cap E^{\lambda^+}_{\neq\cf(\lambda)}$.
    In addition, by Claim~\ref{3184}, there exists no closed witness to
    $\U(\lambda^+,\lambda^+,\theta,\omega)$, and hence by Theorem~\ref{l23s},
    for every stationary $S\s\lambda^+$, we know that $\Tr(S)\cap S^*$ is stationary.
    Consequently, $\refl({<}\cf(\lambda),\lambda^+)$ holds.
  \end{proof}

  \medskip

  \textbf{Case 2: Countable cofinality.} Assume now that $\cf(\lambda) = \omega$.
  Let $\chi := \theta^+$ and $\Delta:=E^{\lambda^+}_\chi$.
  By a result of Eisworth \cite[\S5]{MR2652193}, we obtain a strictly
  increasing sequence of regular cardinals $\langle \lambda_m\mid m<\omega\rangle$
  that converges to $\lambda$ and two matrices, $\vec C=\langle C^m_\alpha\mid
  \alpha<\lambda^+, ~ m<\omega\rangle$ and $\vec e=\langle e_\delta^m \mid
  \delta\in \Delta, ~ m<\omega\rangle$, such that
  \begin{itemize}
    \item for all $m < \omega$, $C^m_0=\emptyset$;
    \item for all $\alpha<\lambda^+$ and $m < \omega$, $C^m_{\alpha+1}=\{\alpha\}$;
    \item for all $\delta\in\Delta$, $\langle e^m_\delta\mid m<\omega\rangle$
      is a $\subseteq$-increasing sequence of club subsets of $\delta$;
    \item for all $\alpha\in\acc(\lambda^+)$, $\langle C^m_\alpha\mid m<\omega\rangle$ is a
      $\subseteq$-increasing sequence of club subsets of $\alpha$;
    \item for all $\alpha\in\acc(\lambda^+)$ and $m < \omega$, $|C^m_\alpha|\le\max\{\lambda_m,\cf(\alpha)\}$;
    \item for all $\alpha\in\acc(\lambda^+)$, $m < \omega$, and $\delta\in(\acc(C^m_\alpha)\cup\{\alpha\})
      \cap \Delta$, $e^m_\delta\s C^m_\alpha$;
    \item for every club $D$ in $\lambda^+$, there exists $\delta\in\Delta$
      such that $\sup(e^m_\delta\cap D\cap E^{\delta}_{>\lambda_m})=\delta$ for all $m<\omega$.
  \end{itemize}
  Define $\mathcal I\s\mathcal P(\lambda^+)$ by letting $A\in\mathcal P(\lambda^+)$
  be in $\mathcal I$ iff there exists a club $D\s\lambda^+$ such that for every
  $\delta\in\Delta\cap D$, for a tail of  $m<\omega$, we have
  $\sup(e^m_\delta\cap D\cap E^{\delta}_{>\lambda_m}\cap A)<\delta$.

  \begin{claim}\label{3191}
    $\mathcal{I}$ satisfies all of the following properties:
    \begin{enumerate}
      \item[(a)] $\mathcal I$ is a proper ideal over $\lambda^+$, extending $\ns_{\lambda^+}$;
      \item[(b)] $\mathcal I$ is $\tau$-indecomposable for all
        $\tau\in\reg(\lambda)\setminus\{\omega,\chi\}$;
      \item[(c)] for all $\sigma<\lambda$, $E^{\lambda^+}_{\ge\sigma}\in\mathcal I^*$.
    \end{enumerate}
  \end{claim}

  \begin{proof}
    (b) Suppose that $\tau\in\reg(\lambda)\setminus\{\omega,\chi\}$ and that
    $\vec A=\langle A_j\mid j<\tau\rangle$ is a $\s$-increasing sequence of elements
    from $\mathcal I$. We shall show that $A:=\bigcup_{j<\tau}A_j$ is in $\mathcal I$, as well.
    For each $j<\tau$, pick a witnessing club $D_j$. We claim that the club
    $D:=\bigcap_{j<\tau}D_j$ witnesses that $A\in\mathcal I$. To see this, let
    $\delta\in\Delta\cap D$ be arbitrary.
    For each $m<\omega$ such that  $\sup(e^m_\delta\cap D\cap E^{\delta}_{>\lambda_m}\cap A)=\delta$,
    as $\vec A$ is $\s$-increasing and $\cf(\delta)\neq\cf(\tau)$,
   we may find $j_m<\tau$ such that
  $\sup(e^m_\delta\cap D\cap E^{\delta}_{>\lambda_m}\cap A_{j_m})=\delta$.
    As $\vec A$ is $\s$-increasing and $\cf(\tau)>\omega$, it follows that there exists a large enough $j<\omega$ such that, for all $m<\omega$,
    if $\sup(e^m_\delta\cap D\cap E^{\delta}_{>\lambda_m}\cap A)=\delta$,
    then
  $\sup(e^m_\delta\cap D\cap E^{\delta}_{>\lambda_m}\cap A_{j})=\delta$.
   But $D\s D_j$, and hence
  $\sup(e^m_\delta\cap D\cap E^{\delta}_{>\lambda_m}\cap A_{j})<\delta$
  for a tail of $m<\omega$. So,
   $\sup(e^m_\delta\cap D\cap E^{\delta}_{>\lambda_m}\cap A)=\delta$
  for a tail of $m<\omega$.
  \end{proof}

  For each $m<\omega$, let $\tr^m(\cdot,\cdot)$, $\rho_2^m(\cdot,\cdot)$ and
  $\lambda_2^m(\cdot,\cdot)$ denote the respective characteristic functions derived
  from walking along the $C$-sequence $\langle C_\alpha^m\mid \alpha<\lambda^+\rangle$.
  Note that for all $\alpha<\beta<\lambda^+$, there is $n < \omega$ such that, for
  every integer $m \geq n$, we have $\tr^m(\alpha, \beta) = \tr^n(\alpha, \beta)$
  (cf.~\cite[p.~1094]{paper13}).

  \begin{claim}\label{3192}
    Suppose that $\mathcal A\s[\lambda^+]^{<\omega}$ is a family consisting of
    $\lambda^+$-many pairwise disjoint sets, $D$ is a club in $\lambda^+$, and
    $\Gamma\in\mathcal{I}^+$. Then there exist $\gamma\in D\cap\Gamma$, $a \in \mathcal{A}$,
    $\epsilon < \gamma$, and $k < \omega$ such that
    \begin{itemize}
      \item $\gamma < a$;
      \item for all $\alpha\in(\epsilon,\gamma)$ and all $\beta \in a$, we have
        $\gamma\in\im(\tr^k(\alpha,\beta))$.
    \end{itemize}
  \end{claim}

  \begin{proof}
    Since $\Gamma\notin\mathcal I$, we may fix $\delta\in\Delta$ such
    that $\sup(e^m_\delta\cap D\cap E^{\delta}_{>\lambda_m}\cap\Gamma)=\delta$
    for cofinally many $m<\omega$. Fix an arbitrary $a\in\mathcal A$ with $a>\delta$.
    Since $a$ is finite, we may find an $n<\omega$ large enough so that, for every $\beta\in a$
    and every integer $m \geq n$, we have $\tr^m(\delta, \beta) = \tr^n(\delta, \beta)$.

    Consider the finite set $T:=\{\delta\} \cup \bigcup_{\beta\in a}\im(\tr^n(\delta,\beta))$,
    and then find an integer $k>n$ such that $\max\{\cf(\tau)
    \mid \tau\in T\}\le\lambda_k$ and $\sup(e^k_\delta\cap D\cap
    E^{\delta}_{>\lambda_k}\cap A)=\delta$. Finally, pick $\gamma\in e^k_\delta\cap
    D\cap E^{\delta}_{>\lambda_k}$ above $\Lambda:=\sup\{\lambda^k_2(\delta,\beta)
    \mid \beta\in a\}$. Let $C:=\bigcup\{ C^k_\tau\mid \tau\in T\}$.
    For all $\tau\in T$, we have $|C^k_\tau|\le\max\{\lambda_k,\cf(\tau)\}=\lambda_k<\cf(\gamma)$,
    and hence $\epsilon:=\max\{\Lambda,\sup(C\cap\gamma)\}$ is less than $\gamma$.

    We claim that $\gamma$, $a$, $\epsilon$, and $k$ are as desired. We clearly
    have $\gamma < a$. To finish, fix an arbitrary $\beta \in a$ and $\alpha\in(\epsilon,\gamma)$.
    We have
    $$
      \lambda^k_2(\delta,\beta)\le\Lambda\le\epsilon<\alpha<\gamma<\delta<\beta,
    $$
    so, by Lemma~\ref{lambda2}, $\tr^k(\delta,\beta)\sq\tr^k(\alpha,\beta)$.
    Set $\ell:=\rho_2^k(\delta,\beta)$. There are now two cases to consider.

    $\br$ If $\delta\in\nacc(C^k_{\tr^k(\delta,\beta)(\ell-1)})$, then, since
    $$
      [\alpha,\delta)\cap C^k_{\tr^k(\delta,\beta)(\ell-1)}\s(\lambda_2^k(\delta,
      \beta),\delta)\cap C^k_{\tr^k(\delta,\beta)(\ell-1)}=\emptyset,
    $$
    we have $\tr^k(\alpha,\beta)(\ell)=\min(C^k_{\tr^k(\alpha,\beta)(\ell-1)}
    \setminus\alpha)=\min(C^k_{\tr^k(\delta,\beta)(\ell-1)}\setminus\alpha)=\delta$.
    It follows that $\gamma\in e^k_\delta\s C^k_\delta\s C$, so
    $$
      [\alpha,\gamma)\cap C^k_{\tr^k(\alpha,\beta)(\ell)}\s (\epsilon,\gamma)\cap C=\emptyset,
    $$
    and hence $\tr^k(\alpha,\beta)(\ell+1)=\min(C^k_{\tr^k(\alpha,\beta)(\ell)}\setminus\alpha)=\gamma$.

    $\br$ If $\delta\in\acc(C^k_{\tr^k(\delta,\beta)(\ell-1)})$, then, since $\delta\in\Delta$,
    we observe that $\gamma\in e^k_\delta\s C^k_{\tr^k(\delta,\beta)(\ell-1)}\s C$.
    It follows that $[\alpha,\gamma)\cap  C^k_{\tr^k(\delta,\beta)(\ell-1)}=\emptyset$,
    and hence $\tr^k(\alpha,\beta)(\ell)=\min(C^k_{\tr^k(\delta,\beta)(\ell-1)}\setminus\alpha)=\gamma$.
  \end{proof}

  We now show that the existence of a closed witness to $\U(\lambda^+,
  \lambda^+, \theta, \omega)$ follows from each of the hypotheses identified
  in the statement of the theorem. Note first that condition (2) is trivially
  taken care of, as there are no infinite cardinals strictly less
  than $\omega = \cf(\lambda)$. The next claim will deal with
  conditions (3) and (4).

  \begin{claim}\label{3193}
    Suppose that $c:[\lambda^+]^2\rightarrow\theta$ is a coloring,
    $\sigma\in\reg(\lambda)$, and one of the following two conditions holds:
    \begin{itemize}
      \item $c$ is a somewhere-closed witness to $\U(\lambda^+,2,\theta,\omega)$; or
      \item $c$ is a $E^{\lambda^+}_{\ge\sigma}$-closed witness to $\U(\lambda^+,2,\theta,2)$.
    \end{itemize}
    Then there exists a closed  witness to $\U(\lambda^+,\lambda^+,\theta,\omega)$.
  \end{claim}

  \begin{proof}
    Define $d:[\lambda^+]^2\rightarrow\theta$ by setting, for all $\alpha < \beta < \lambda^+$,
    $$
      d(\alpha,\beta):=\max\left\{ c(\zeta,\gamma)\Mid(\zeta,\gamma)\in
      \left[\bigcup_{m<\omega}\im(\tr^m(\alpha,\beta))\right]^2\right\},
    $$
    provided that the set is nonempty, and $d(\alpha,\beta):=0$, otherwise.

    We claim that $d$ is as desired. To see that $d$ is closed, suppose that
    $\beta < \lambda^+$, $i < \theta$, and $A \subseteq D^d_{\leq i}(\beta)$,
    with $\eta := \sup(A)$ smaller than $\beta$. To show that $\eta \in D^d_{\leq i}(\beta)$,
    fix $n < \omega$ large enough so that $\{ \tr^m(\eta,\beta)\mid m<\omega\}=
    \{ \tr^m(\eta,\beta)\mid m<n\}$, and then fix $\alpha \in A$ above
    $\max_{m<n}\lambda^m_2(\eta,\beta)$. By Lemma~\ref{lambda2}, $\bigcup_{m<n}
    \im(\tr^m(\eta,\beta))\s\bigcup_{m<\omega}\im^m(\tr(\alpha,\beta))$, and hence,
    by the definition of $d$, we have $d(\eta,\beta)\le d(\alpha,\beta)\le i$.

    To see that $d$ witnesses $\U(\lambda^+,\lambda^+,\theta,\omega)$,
    it suffices to verify Clause~(2) of Lemma~\ref{pumpclosed}. To this end,
    suppose that $\mathcal A\s[\lambda^+]^{<\omega}$ is a family consisting of
    $\lambda^+$-many pairwise disjoint sets, $D$ is a club in $\lambda^+$, and $i<\theta$.
    We shall prove that there exist $\zeta\in D$, $a\in\mathcal A$, and
    $\epsilon^* < \zeta$ for which
    \begin{itemize}
      \item $\zeta < a$;
      \item for all $\alpha \in (\epsilon^*, \zeta)$ and all $\beta\in a$,
        we have $d(\alpha,\beta)>i$.
    \end{itemize}
    Let $\Gamma$ be the set of $\gamma \in E^{\lambda^+}_{\geq \sigma}$ for
    which there exist $a \in \mathcal{A}$, $\epsilon < \gamma$, and $k < \omega$
    such that
    \begin{itemize}
      \item $\gamma < a$;
      \item for all $\beta \in a$ and $\alpha \in (\epsilon, \gamma)$, we have
        $\gamma \in \im(\tr^k(\alpha, \beta))$.
    \end{itemize}
    By Claim~\ref{3192} and Claim~\ref{3191}(3), $\Gamma$ is stationary.
    For each $\gamma\in\Gamma$, pick $a_\gamma\in\mathcal A$, $\epsilon_\gamma<\gamma$
    and $k_\gamma<\omega$ witnessing that $\gamma\in\Gamma$. Fix a
    stationary subset $\Gamma'\s\Gamma$ on which the map $\gamma\mapsto(\epsilon_\gamma,k_\gamma)$
    is constant, with value, say, $(\epsilon,k)$.

    Now, let $S$ be the set of $\varsigma < \lambda^+$ for which there exist
    $\gamma \in \Gamma'$ and $\varepsilon < \varsigma$ such that
    \begin{itemize}
      \item $\varsigma < \gamma$;
      \item for all $\zeta \in (\varepsilon, \varsigma)$, we have $c(\zeta, \gamma) > i$.
    \end{itemize}
    We claim that $S$ is stationary. There are two cases to consider.

    $\br$ If $c$ is a somewhere-closed witness to $\U(\lambda^+,2,\theta,\omega)$,
    then, by the implication $(1)\implies(2)$ of Lemma~\ref{pumpclosed},
    we infer that $S$ is stationary.

    $\br$ If $c$ is a $E^{\lambda^+}_{\ge\sigma}$-closed witness to $\U(\lambda^+,2,\theta,2)$,
    then repeating the proof of the implication $(1)\implies(2)$ of Lemma~\ref{pumpclosed}
    in the current setting implies that, furthermore, $S\cap\Gamma'$ is stationary.

    For each $\varsigma\in S$, pick $\gamma_\varsigma\in\Gamma'$ and
    $\varepsilon_\varsigma<\varsigma$ witnessing that $\varsigma\in S$.
    Fix a stationary subset $S'\s S$ on which the map $\varsigma\mapsto\varepsilon_\varsigma$
    is constant, with value, say, $\varepsilon$.

    \begin{subclaim}
      There exist $\zeta\in D$, $\varsigma\in S'$, $\eta<\zeta$ and $l<\omega$ such that
      \begin{itemize}
        \item $a_{\gamma_\varsigma}>\gamma_\varsigma>\varsigma>\zeta>\max\{\epsilon,\varepsilon\}$;
        \item for all $\beta\in a_{\gamma_\varsigma}$ and $\alpha\in(\eta,\zeta)$,
          we have $\zeta\in\im(\tr^l(\alpha,\beta))$.
      \end{itemize}
    \end{subclaim}

    \begin{proof}
      The proof is nearly identical to that of Claim~\ref{3192}. Fix $\delta\in\Delta$
      above $\max\{\epsilon, \varepsilon \}$ such that
      $\sup(e^m_\delta\cap D\cap E^{\delta}_{>\lambda_m})=\delta$ for all $m<\omega$.
      Fix $\varsigma\in S'$ above $\delta$, and let $a:=a_{\gamma_\varsigma}$.
      Find $n<\omega$ large enough so that, for every
      $\beta\in a$ and every integer $m \geq n$, we have $\tr^m(\delta, \beta) =
      \tr^n(\delta, \beta)$. Let $T:=\{\delta\} \cup \bigcup_{\beta\in a}\im(\tr^n(\delta,\beta))$,
      and find an integer $l>n$ large enough so that, for all $\tau\in T$, we have
      $\cf(\tau)\le\lambda_l$. Let $\Lambda:=\sup\{\lambda^l_2(\delta,\beta)\mid \beta\in a\}$,
      and pick $\zeta\in e^l_\delta\cap D\cap E^{\lambda^+}_{>\lambda_l}$
      above $\max\{\Lambda,\epsilon,\varepsilon\}$. Let $C:=\bigcup\{ C^l_\tau\mid \tau\in T\}$.
      For all $\tau\in T$, we have $|C^l_\tau|\le\max\{\lambda_l,\cf(\tau)\}=\lambda_l<\cf(\zeta)$,
      and hence $\eta:=\max\{\Lambda,\sup(C\cap\zeta)\}$ is less than $\zeta$.

      We claim that $\zeta$, $\varsigma$, $\eta$, and $l$ are as desired.
      The first requirement is clearly satisfied. To verify the second, fix an
      arbitrary $\beta \in a$ and $\alpha\in(\eta,\zeta)$. We have
      $$
        \lambda^l_2(\delta,\beta)\le\Lambda\le\eta<\alpha<\zeta<\delta<\varsigma<\gamma_\varsigma<\beta,
      $$
      so, by Lemma~\ref{lambda2}, $\rho_2^l(\delta,\beta)\sq\rho_2^l(\alpha,\beta)$.
      Let $\ell:=\rho_2^l(\delta,\beta)$. Then, as in the proof of Claim~\ref{3192},
      we infer that $\zeta\in\{\tr^l(\alpha,\beta)(\ell), ~ \tr^l(\alpha,\beta)(\ell+1)\}$.
    \end{proof}

    Let $\zeta$, $\varsigma$, $\eta$, and $l$ be given by the preceding subclaim.
    Let $\gamma:=\gamma_\varsigma$, $a:=a_{\gamma}$, and $\epsilon^* := \max\{\epsilon,\varepsilon,\eta\}$.
    We claim that $\zeta$, $a$, and $\epsilon^*$ are as sought. To prove this,
    let $\alpha \in (\epsilon^*, \zeta)$ and $\beta \in a$ be arbitrary.
    As $\gamma \in \Gamma'$, $\beta \in a = a_\gamma$, and
    $\alpha \in (\epsilon, \gamma)$, we have $\gamma \in \im(\tr^k(\alpha, \beta))$.
    As $\alpha \in (\eta, \zeta)$, we have $\zeta \in \im(\tr^l(\alpha, \beta))$.
    Finally, as $\varsigma \in S'$ and $\zeta \in (\varepsilon, \varsigma)$,
    we have $c(\zeta, \gamma) > i$. Altogether, we obtain $d(\alpha, \beta)
    > i$, as desired.
  \end{proof}

  We now finish the proof of the countable cofinality case and hence
  the theorem by disposing with condition (1).

  \begin{claim}
    Suppose that there exists no closed witness to $\U(\lambda^+,\lambda^+,\theta,\omega)$.
    Then $\refl({<}\omega,\lambda^+)$ holds.
  \end{claim}

  \begin{proof}
    The proof of Claim~\ref{3183} makes it clear that Claim~\ref{3192} implies the
    existence of a closed witness to $\U(\lambda^+,\lambda^+,\theta,\omega)$,
    provided that $\mathcal I$ is not weakly $\theta$-saturated.
    Consequently, $\mathcal I$ is weakly $\theta$-saturated.
    By Corollary~\ref{predecessor}, $\theta\neq\omega$. Altogether,
    $\theta\in\reg(\lambda)\setminus\{\omega,\chi\}$, so,
    by Claim~\ref{3191}, $\mathcal I$ is an ideal that is weakly $\theta$-saturated
    and $\theta$-indecomposable. It then follows from \cite[Theorem~2(4)]{MR2652193}
    that $\refl({<}\omega,S^*)$ holds, where $S^*:=E^{\lambda^+}_{\ge\theta}\cap
    E^{\lambda^+}_{\neq\omega}$. In addition, by Theorem~\ref{l23s},
    for every stationary $S\s\lambda^+$, we know that $\Tr(S)\cap S^*$ is stationary.
    Therefore, $\refl({<}\omega,\lambda^+)$ holds.
  \end{proof}

\subsection{Inaccessible cardinals} \label{inaccessible_subsection}

We begin this subsection by noting the following result. It follows immediately
from the proof of Theorem~\ref{l23s}, so we do not provide a separate proof here.

\begin{prop}
  Suppose that $\kappa$ is an inaccessible cardinal, $\theta, \chi\in\reg(\kappa)$,
  and there exists a stationary $S\s E^\kappa_\chi$ such that $\clubsuit(S)$ holds and
  $$
    \sup\{ \nu<\kappa\mid  \Tr(S)\cap E^{\kappa}_\nu\text{ is stationary}\}<\kappa.
  $$
  Then there exists a closed witness to $\U(\kappa,\kappa,\theta,\chi)$. \qed
\end{prop}

Our last result of this section, similarly to Theorem~\ref{l23s}, provides an
improvement to the implication $(3)\implies(4)$ from Corollary~\ref{l23}, this
time in the context of inaccessible cardinals.

\begin{thm} \label{inaccessible_reflection}
  Suppose that $\kappa$ is an inaccessible cardinal, $\theta, \chi\in\reg(\kappa)$, and
  there exists a stationary $S\s E^{\kappa}_{\ge\chi}$ that does not reflect at any
  inaccessible cardinal. Then there exists a closed witness to $\U(\kappa,\kappa,\theta,\chi)$.
\end{thm}

\begin{proof}
  By Corollary~\ref{l23s}, we may assume that, for every stationary
  $T\s E^\kappa_{\ge\chi}$, the set $\Tr(T)\cap E^\kappa_{\ge\theta}$ is stationary.
  We begin by isolating stationary sets as in the statement of the theorem
  that are slightly better-behaved.

  \begin{claim}\label{c3201}
    There exist regular cardinals $\sigma,\tau$ with $\max\{\aleph_1,\chi,\theta\}\le\sigma<\tau<\kappa$
    and stationary subsets $S,S^0$ of $\kappa$ such that
    \begin{itemize}
      \item $S\s E^\kappa_{\sigma}\cap\card$, and $S$ does not
        reflect at inaccessibles;
      \item $S^0\s E^\kappa_\tau$, and $S^0$ does not reflect at inaccessibles.
    \end{itemize}
  \end{claim}

  \begin{proof}
    By the hypothesis of the theorem, we can fix a stationary $T\s E^{\kappa}_{\ge\chi}$
    such that $T$ does not reflect at inaccessibles. Then $\Tr(T)\cap E^\kappa_{\ge\theta}$
    is a stationary set consisting of singular ordinals, so Fodor's lemma
    entails the existence of a cardinal $\sigma\in\reg(\kappa)\setminus\theta$ for which
    $\Tr(T)\cap E^\kappa_{\sigma}$ is stationary. Since $\card \cap \kappa$ is a club
    in the inaccessible $\kappa$, $S:=\Tr(T)\cap\card\cap E^\kappa_\sigma$ is a
    stationary subset of $E^\kappa_{>\chi}$. As $\Tr(S) \s \Tr(T)$, we can repeat the process to find
    $\tau \in \reg(\kappa)$ such that $S^0 := \Tr(S) \cap E^\kappa_\tau$
    is stationary. Then $\tau>\sigma>\chi\ge\aleph_0$, $\sigma\ge\theta$ and $\Tr(S^0)\s\Tr(S)\s\Tr(T)$,
    so $\sigma$, $\tau$, $S$, and $S^0$ are as sought.
  \end{proof}

  Let $\sigma$, $\tau$, $S$, and $S^0$ be given by the preceding claim.
  By \cite[Theorem~2.1.1]{MR3321938}, there exists a sequence
  $\langle e_\delta\mid\delta\in S\rangle$ such that
  \begin{itemize}
    \item for all $\delta\in S$, $e_\delta$ is a club in $\delta$ of order type $\sigma$;
    \item for all $\delta\in S$, $\langle \cf(\gamma)\mid \gamma\in\nacc(e_\delta)\rangle$
      is strictly increasing and converges to $\delta$;
    \item for every club $D\s\kappa$, there exists $\delta\in S$ with $e_\delta\s D$.
  \end{itemize}
  Define $\mathcal I\s\mathcal P(\kappa)$ by letting $A\in\mathcal P(\kappa)$ be in
  $\mathcal I$ iff there exists a club $D\s\kappa$ such that for every
  $\delta\in S\cap\acc(D)$, $\sup(\nacc(e_\delta)\cap D\cap A)<\delta$.

  \begin{claim}\label{3202}
    $\mathcal{I}$ satisfies the following two conditions:
    \begin{enumerate}
      \item $\mathcal I$ is a $\sigma$-complete proper ideal over $\kappa$, extending $\ns_{\kappa}$;
      \item $\mathcal I$ is not weakly $\theta$-saturated.
    \end{enumerate}
  \end{claim}

  \begin{proof}
    Clause~(1) is straightforward to verify. To see that Clause~(2) holds, we shall
    want to appeal to \cite[Claim~3.3]{Sh:365}. For each $\delta\in S$, let
    $I_\delta:=\{ A\s e_\delta\mid \sup(\nacc(e_\delta)\cap A)<\delta\}$, so that
    $I_\delta$ is a $\sigma$-complete and $\tau$-indecomposable ideal over $e_\delta$.
    Trivially, $\sup_{\delta\in S}|e_\delta|^+<\kappa$. Setting $\bar C:=\langle
    e_\delta\mid\delta\in S\rangle$ and $\bar I:=\langle I_\delta\mid \delta\in S\rangle$,
    and recalling \cite[Definition~3.1]{Sh:365}, it is evident that the ideal
    $\id_p(\bar C,\bar I)$ is equal to our proper ideal $\mathcal I$.
    Now, since $S^0$ is a stationary subset of $E^\kappa_\tau$ that does not reflect at
    inaccessibles, Case $(\beta)(a)$ of \cite[Claim~3.3]{Sh:365} entails the existence
    of a partition of $\kappa$ into $\tau$-many $\mathcal I$-positive sets.
    In particular, since $\tau>\theta$, $\mathcal I$ is not weakly $\theta$-saturated.
  \end{proof}

  Using the preceding claim, fix a surjection $h:\kappa\rightarrow\theta$ such that
  $h^{-1}\{i\}\in\mathcal I^+$ for all $i<\theta$. Next, using \cite[Proposition~4.3.1]{MR3321938}
  and the fact that $S \s \card$ and $S$ does not reflect at inaccessibles, fix a $C$-sequence
  $\vec C=\langle C_\alpha\mid\alpha<\kappa\rangle$ such that
  \begin{itemize}
    \item for all $\alpha\in\reg(\kappa)$, $C_\alpha$ is a club in $\alpha$
      disjoint from $S$;
    \item for all $\alpha\in\acc(\kappa)\setminus\reg(\kappa)$, $C_\alpha$ is a
      club in $\alpha$ satisfying:
      \begin{itemize}
        \item $|C_\alpha|<\min(C_\alpha)$;
        \item $|C_\alpha|\le\max\{\sigma,\cf(\alpha)\}$;
        \item for all $\delta\in(C_\alpha\cup\{\alpha\})\cap S$, $\sup(e_\delta\setminus C_\alpha)<\delta$.
      \end{itemize}
  \end{itemize}

  We shall walk along $\vec C$. Derive a closed coloring $c:[\kappa]^2\rightarrow\theta$
  as in the proof of Theorem~\ref{simu} by setting, for all $\alpha < \beta < \kappa$,
  $$
    c(\alpha,\beta):=\max\{ h(\xi)\mid \xi\in\im(\tr(\alpha,\beta))\}.
  $$
  We claim that $c$ witnesses $\U(\kappa, \kappa, \theta, \chi)$ and prove this by
  verifying Clause~(2) of Lemma~\ref{pumpclosed}. To this end, fix a family
  $\mathcal{A} \s [\kappa]^{<\chi}$ consisting of $\kappa$-many pairwise disjoint
  sets, a club $D$ in $\kappa$, and a color $i < \theta$. We will find $\gamma \in D$,
  $a \in \mathcal{A}$, and $\epsilon < \gamma$ such that
  \begin{itemize}
    \item $\gamma < a$;
    \item for all $\alpha \in (\epsilon, \gamma)$ and all $\beta \in a$,
      we have $c(\alpha, \beta) > i$.
  \end{itemize}

  Since $\Gamma:=h^{-1}\{i+1\} \setminus (\tau + 1)$ is in $\mathcal I^+$, we may
  fix $\delta\in S \setminus (\tau + 1)$ such that $\sup(\nacc(e_\delta)\cap
  D\cap\Gamma)=\delta$. Fix an arbitrary $a\in\mathcal A$ with $\delta < a$, and set
  \begin{itemize}
    \item $T:=\{\xi\in \bigcup_{\beta\in a}\im(\tr(\delta,\beta))\mid \delta\in C_\xi\}$;
    \item $C:=C_\delta\cup\bigcup\{ C_\xi\mid \xi\in T\}$;
    \item $\Lambda:=\sup\{\lambda^k_2(\delta,\beta), ~ \sup(e_\delta\setminus C_\xi)
      \mid \beta\in a, ~ \xi\in T\}$.
  \end{itemize}

  For all $\xi\in T$, since $C_\xi\cap S\neq\emptyset$, we infer that
  $\xi\notin\reg(\kappa)$, $\sup(e_\delta\setminus C_\xi)<\delta$, and
  $|C_\xi|<\min(C_\xi)<\delta$. In addition, $|C_\delta|=\sigma<\tau<\delta$
  and $|a|<\chi\le\sigma=\cf(\delta)$, so it follows that both $|C|$ and $\Lambda$
  are less than $\delta$. Pick $\gamma\in\nacc(e_\delta)\cap D\cap\Gamma$ large enough so that
  $\gamma>\Lambda$ and $\cf(\gamma)>\max\{|C|,\chi\}$, and hence
  $\epsilon:=\max\{\Lambda, ~ \sup(C\cap\gamma)\}$ is less than $\gamma$.

  We claim that $\gamma$, $a$, and $\epsilon$ are as desired. To this end,
  let $\alpha \in (\epsilon, \gamma)$ and $\beta \in a$ be arbitrary. We have
  $$
    \lambda_2(\delta,\beta)\le\Lambda\le\epsilon<\alpha<\gamma<\delta<\beta,
  $$
  so, by Lemma~\ref{lambda2}, $\rho_2(\delta,\beta)\sq\rho_2(\alpha,\beta)$.
  We claim that $\gamma \in \im(\tr(\alpha, \beta))$.
  Set $\ell:=\rho_2(\delta,\beta)$, and consider the following two cases.

  $\br$ If $\delta\in\nacc(C_{\tr(\delta,\beta)(\ell-1)})$, then, since
  $$
    [\alpha,\delta)\cap C_{\tr(\delta,\beta)(\ell-1)}\s(\lambda_2(\delta,\beta),\delta)
    \cap C_{\tr(\delta,\beta)(\ell-1)}=\emptyset,
  $$
  we have $\tr(\alpha,\beta)(\ell)=\min(C_{\tr(\alpha,\beta)(\ell-1)}\setminus\alpha)=
  \min(C_{\tr(\delta,\beta)(\ell-1)}\setminus\alpha)=\delta$. It follows that
  $C_{\tr(\alpha,\beta)(\ell)} = C_\delta$, so, since $\gamma \in e_\delta
  \setminus (\Lambda + 1) \s C_\delta \s C$, we have
  $$
    [\alpha,\gamma)\cap C_{\tr(\alpha,\beta)(\ell)}\s (\epsilon,\gamma)\cap C=\emptyset
  $$
  and $\tr(\alpha,\beta)(\ell+1)=\min(C_{\tr(\alpha,\beta)(\ell)}\setminus\alpha)=\gamma$.

  $\br$ If $\delta\in\acc(C_{\tr(\delta,\beta)(\ell-1)})$, then $\gamma\in
  e_\delta\setminus(\Lambda + 1) \s C_{\tr(\delta,\beta)(\ell-1)}\s C$, so
  $[\alpha,\gamma)\cap  C_{\tr(\delta,\beta)(\ell-1)}=\emptyset$
  and $\tr(\alpha,\beta)(\ell)=\min(C_{\tr(\delta,\beta)(\ell-1)}\setminus\alpha)=\gamma$.

  So, in either case, $\gamma \in \im(\tr(\alpha, \beta))$, and hence
  $c(\alpha,\beta)\ge h(\gamma)>i$, as desired.
\end{proof}

\section{Concluding remarks}
\begin{enumerate}
\item Theorem~1 of \cite{paper13} states that if $\lambda$ is a singular cardinal,
$\theta\le\lambda^+$, and $\pr_1(\lambda^+,\lambda^+,\theta,\chi)$ holds for $\chi=2$,
then $\pr_1(\lambda^+,\lambda^+,\theta,\chi)$ holds also for $\chi=\cf(\lambda)$.
Theorem~\ref{t310} above implies that the latter is optimal and cannot be improved to $\chi=\cf(\lambda)^+$.
Specifically, if $\lambda$ is a singular limit of strongly compact cardinals,
then $\pr_1(\lambda^+,\lambda^+,\lambda^+,2)$ holds,\footnote{By \cite{ehr}, if $2^\lambda=\lambda^+$, then $\pr_1(\lambda^+,\lambda^+,\lambda^+,2)$ holds.}
but $\pr_1(\lambda^+,\lambda^+,\lambda^+,\cf(\lambda)^+)$ fails.
\item In light of Lemma~\ref{lemma24} and Theorem~\ref{increasechi}(3), we ask whether $\U(\lambda^+,2,\theta,2)$
for every pair of infinite cardinals $\theta \leq \lambda$, the instance $\U(\lambda^+,2,\theta,2)$
implies $\U(\lambda^+,\lambda^+,\theta,\cf(\lambda))$.
\item We do not know whether it is the case that, in $\zfc$, any true instance $\U(\kappa,\kappa,\ldots)$ may be witnessed by a closed coloring.
\item We wonder whether Subsection~\ref{inaccessible_subsection} can be expanded
to say more on instances of $\U(\kappa,\kappa,\ldots)$ in which $\kappa$ is
an inaccessible cardinal of the form $\cf(2^\nu)$.
\item In view of Fact~\ref{fact29}, we conjecture that $\kappa$ is weakly compact
iff $\U(\kappa, 2, \theta, 2)$ fails for all $\theta \in\reg(\kappa)$.
Recalling \cite[Question~8.1.4]{todorcevic_book}, we furthermore
conjecture that $\kappa$ is weakly compact iff $\U(\kappa, 2, \omega, 2)$ fails.
\end{enumerate}

\section*{Acknowledgments}
The results of this paper were presented by the first author at the
\emph{Set Theory, Model Theory and Applications} conference in Eilat, April 2018,
and at the \emph{SETTOP 2018} conference in Novi Sad, July 2018,
and by the second author at the \emph{$11^{\text{th}}$ Young Set Theory} workshop in Lausanne,
June 2018. We thank the organizers for the warm hospitality.

\end{document}